\documentclass[12pt]{amsart}

\usepackage{graphics,color,tikz,mathrsfs,amscd,eucal,mathabx,verbatim,booktabs}
\usepackage{graphicx} % Allows including images
\usepackage{booktabs} % Allows the use of \toprule, \midrule and \bottomrule in tables
\usepackage{amscd}
\usepackage{verbatim}
\usepackage{xcolor}

\usetikzlibrary{calc}

\numberwithin{equation}{section}
\numberwithin{table}{section}

\theoremstyle{plain}
\newtheorem{theorem}{Theorem}[section]
\newtheorem{lemma}{Lemma}[section]
\newtheorem{prop}{Proposition}[section]

\theoremstyle{remark}
\newtheorem{remark}{Remark}[section]

\newtheorem{definition}{Definition}[section]

\newcommand{\Kt}{\tilde{K}}

\newcommand{\Nedelec}{N{\'{e}}d{\'{e}}lec }

\newcommand{\diam}{\mathop\mathrm{diam}}
\newcommand{\dett}{\mathop\mathrm{det}}
\newcommand{\dive}{\mathop\mathrm{div}}
\newcommand{\grad}{\ensuremath{\mathop{{\bf{grad}}}}}
\newcommand{\gradrz}{\ensuremath{{\mathbf{grad}}}_{rz}}

\newcommand{\bcurl}{\mathop{\mathbf{curl}}}

\newcommand{\curlrz}{\mathrm{curl}_{rz}}
\newcommand{\bcurlrz}{\mathbf{curl}_{rz}}

\newcommand{\TT}{\mathscr{T}}

\newcommand{\fe}{{\mathfrak{e}}}

\newcommand{\ba}{{\boldsymbol{a}}}

\newcommand{\bB}{{\boldsymbol{B}}}

\newcommand{\bC}{{\boldsymbol{C}}}

\newcommand{\bd}{{\boldsymbol{d}}}
\newcommand{\bz}{{\boldsymbol{z}}}

\newcommand{\be}{{\boldsymbol{e}}}

\newcommand{\bn}{{\boldsymbol{n}}}

\newcommand{\bH}{{\boldsymbol{H}}}

\newcommand{\bL}{{\boldsymbol{L}}}

\newcommand{\bu}{{\boldsymbol{u}}}

\newcommand{\bv}{{\boldsymbol{v}}}
\newcommand{\bw}{{\boldsymbol{w}}}

\newcommand{\RRR}{{\mathbb{R}}}

\newcommand{\bt}{{\boldsymbol{t}}}
\newcommand{\bx}{{\boldsymbol{x}}}
\newcommand{\by}{{\boldsymbol{y}}}

\newcommand{\bZ}{{\boldsymbol{Z}}}

\newcommand{\Da}{D_{\ba}}

\newcommand{\triint}{\int_{D_{\ba_1}}\!\!\int_{D_{\ba_2}}\!\!\int_{D_{\ba_3}}}
\newcommand{\trid}{d\boldsymbol{y}_{\!3} d\boldsymbol{y}_{\!2}d\boldsymbol{y}_{\!1}}

\newcommand{\xt}{\tilde{\boldsymbol{x}}_{\boldsymbol{y}}}

\newcommand{\Lrnk}[1]{\| {#1} \|_{L^2_r(K)}}
\newcommand{\byi}{{\boldsymbol{y}_i}}
\newcommand{\bai}{{\boldsymbol{a}_i}}
\newcommand{\baj}{{\boldsymbol{a}_j}}
\newcommand{\bak}{{\boldsymbol{a}_k}}
\newcommand{\bff}{{\boldsymbol{f}}}

\newcommand{\tn}{\textnormal}

\newcommand{\iDai}{\int_{D_{\ba_{1}}}}
\newcommand{\iDaj}{\int_{D_{\ba_{2}}}}
\newcommand{\iDak}{\int_{D_{\ba_{3}}}}

\newcommand{\iDajj}{\int_{D_{\ba_{j}}}}
\newcommand{\iDakk}{\int_{D_{\ba_{k}}}}

\newcommand{\dJ}{\mathop\mathrm{det}\Big(\dfrac{d\tilde{\boldsymbol{x}}_{\boldsymbol{y}}}{d\boldsymbol{x}}\Big)}



\author{Minah Oh} 
\address{James Madison University, Department of Mathematics and Statistics,
  Harrisonburg, VA 22807.}\email{ohmx@jmu.edu}
    \thanks{The work of the author was partially supported by NSF grant number DMS-1913050.}

\title[Projectors in weighted norms]{The Hodge Laplacian on Axisymmetric Domains and its Discretization}

\subjclass[ ]{ }



\begin{document}

\begin{abstract}
We study the mixed formulation of the abstract Hodge Laplacian 
on axisymmetric domains with general data through Fourier finite element methods in weighted function spaces.
Closed Hilbert complexes and commuting projectors are used as in the work of D. N. Arnold, R. S. Falk, and R. Winther, Finite element exterior calculus: from Hodge theory to numerical stability, Bull. Amer. Math.
Soc. (N.S.), 47 (2010), pp. 281--354, by using the new family of finite element spaces for general axisymmetric problems introduced in M. Oh, de Rham complexes arising from Fourier finite element methods in axisymmetric domains,
Comput. Math. Appl., 70 (2015), pp. 2063--2073.
In order to get stability results and error estimates for the discrete mixed formulation,
we construct commuting projectors that can be applied to functions with low regularity.   
\end{abstract}

\maketitle

\section{Introduction} 

An axisymmetric problem is a problem defined on a three-dimensional ($3D$) domain $\breve{\Omega}$ that is symmetric with respect to an axis, i.e.,
$\breve{\Omega}\subset\RRR^3$ is obtained by rotating a two-dimensional ($2D$) domain $\Omega\subset\RRR^2_+=\{(r,z)\in\RRR^2: r\geq 0 \}$ around the axis of symmetry (the $z$-axis).
Throughout this paper, we will assume that $\Omega$ is a bounded Lipschitz domain.
\textcolor{black}{We will also use $(r,\theta,z)$ to denote cylindrical coordinates.} 
An axisymmetric problem with data that is independent of the rotational variable $\theta$ can be reduced to a $2D$ problem 
by using cylindrical coordinates. An axisymmetric problem with data that has $\theta$-dependency can be reduced to a sequence of $2D$ problems
by using cylindrical coordinates and a Fourier series decomposition in the $\theta$-variable ($-\pi\leq\theta\leq\pi$), where the solution to each $2D$ problem
is the $n$-th Fourier mode of the $3D$ solution. A discrete problem corresponding to a $2D$ problem is significantly smaller than that corresponding to a $3D$ problem,
so such dimension reduction is an attractive feature considering computation time. Due to the Jacobian arising from change of variables, however, the resulting $2D$ problems
are posed in weighted functions spaces where the weight function is the radial component $r$. 
Furthermore, the formulas of the $\grad, \bcurl$, and $\dive$ operators affecting the
$n$-th Fourier mode of a function is quite different from the standard ones. In particular, there are multiple $\dfrac{n}{r}$-terms appearing in those formulas, so these operators
do not map standard polynomial spaces into another polynomial space. This makes it difficult to construct finite element spaces that form a discrete de Rham complex even though
that is a standard tool in the study of mixed finite element methods. (See \cite{monk:book, AFW:2006, AFW:2006:part1, AFW:2010}.) This difficulty was overcome in \cite{O:2015} where the author constructed a new family of finite element spaces (call them $A_h, \bB_h, \bC_h$, and $D_h$)
that satisfy the following exact sequence property:
\begin{equation} \label{intro1}
\begin{CD}
0  \rightarrow A_h   @>\mathbf{grad}^n_{rz}>> \bB_h    @>\bcurl^n_{rz}>> \bC_h    @>\dive^n_{rz}>>   D_h \rightarrow 0, 
\end{CD}
\end{equation}
where $\mathbf{grad}^n_{rz}, \bcurl^n_{rz}$, and $\dive^n_{rz}$ are the operators of interest when considering axisymmetric problems with general data.
This paper will use these Fourier finite element spaces to approximate the solution to the mixed formulation of the abstract Hodge Laplacian on axisymmetric domains with general data.

In \cite{O:2014}, the continuous and discrete mixed formulations of the Hodge Laplacian on axisymmetric domains 
were studied through Hilbert complexes under the additional assumption that the given data in the problem is $\theta$-independent. 
In this paper, we extend these results to the general case where the data is dependent on $\theta$. The key ingredients needed to accomplish this task is the
construction of a sequence of de Rham complexes and a uniformly $W$-bounded (or $V$-bounded) cochain projections that can be applied to axisymmetric problems with general data.
In other words, we construct necessary tools for the so-called Fourier-finite-element-methods (Fourier-FEMs) so that we can apply the well-known theory of \cite{AFW:2010} to general axisymmetric problems. 

%In section~\ref{sec:preliminaries}, details of the Fourier finite element methods are given.

Projection operators which commute with the governing differential operators are key tools for the stability analysis of finite element methods (FEMs). 
Commuting projection operators that are well-defined for functions with lower regularity are now standard for the analysis of mixed FEMs as well. 
Mixed formulations on axisymmetric domains have been studied through these commuting projection operators by many authors under the assumption that the data is axisymmetric ($\theta$-independent).
(See \cite{CGP:2008, MR2670109, CGO:2010, Oh_PHD, Ervin:2013, O:2014}.) In \cite{GO:2012}, commuting smoothed projections that can be applied to axisymmetric problems with axisymmetric data  
were constructed by modifying the work of Schoberl (\cite{Schob01, Schberl2005NumericalMF}) and Christiansen and Winther (\cite{Christiansen:2008}) to appropriate weighted functions spaces. In this paper, we modify \cite{GO:2012} to construct commuting smoothed projections onto the
discrete spaces arising in (\ref{intro1}) that can be applied to axisymmetric problems with general data. These projections can then be used as in \cite{AFW:2010} to prove stability and convergence results for the weighted mixed formulation of the Hodge Laplacian. 

The remainder of the paper is organized as follows. In the following section, we give a brief introduction to the Fourier series decomposition and Fourier-FEMs. We also introduce the main Hilbert spaces of interest there along with some notations. 
In section~\ref{section3}, we introduce the abstract Hodge Laplacian on axisymmetric domains and its mixed formulation following the framework of \cite{AFW:2010}. 
In section~\ref{section4}, we summarize the new family of finite element spaces constructed in \cite{O:2015} that can be applied to axisymmetric problems with general data 
and continue to discuss the discrete weighted mixed formulation of the Hodge Laplacian. We state the main stability and quasi-optimality results at the end of this section as well. 
In section~\ref{section5}, we construct 
a so-called W-bounded cochain projections that are needed to prove the main result stated in section~\ref{section4}, and
Numerical results follow in section~\ref{numerical}. Some proofs have been moved to the Appendix (section~\ref{appendix}) to improve readability of the paper. 

\section{Preliminaries}
\label{sec:preliminaries}

Axisymmetric problems with general data have been studied by many authors through a Fourier series decomposition. (See \cite{BDM:1999, Heinrich:1996, Boniface:2005, CL:2011} for example.)
The underlying principle
is the application of partial Fourier approximation (truncated partial Fourier series) using
trigonometric polynomials of degree $N$ with respect to $\theta$. This step reduces the $3D$ problem into
$N$ $2D$ problems. The term Fourier-FEMs is used when each
Fourier mode is approximated by using appropriate FEMs. Fourier-FEMs for the axisymmetric Poisson equations have been studied in \cite{Heinrich:1996}, and in \cite{Nkemzi:2005,Nkemzi:2007},
Fourier-FEMs for the Maxwell equations were analyzed.

If a $3D$ problem is defined on an axisymmetric domain $\breve{\Omega}\subset\RRR^3$, then since any function $u$ defined on $\breve{\Omega}$ is periodic with respect to $\theta$,
one can use a Fourier series decomposition to represent this function. In particular,
by using the orthogonal and complete system 
$\left\{ 1, \sin\theta, \cos\theta, \cdot\cdot\cdot, \sin n\theta, \cos n\theta, \cdot\cdot\cdot \right\}$ in $L^2((-\pi, \pi))$ the function $u$ can be written as
\begin{align} \label{scalar_fourier}
u &= u_0 + \sum_{n=1}^\infty u_n\cos n\theta + \sum_{n=1}^\infty u_{-n}\sin n\theta.
\end{align} 

Similarly, for vector-valued functions on $\breve{\Omega}$, we first write 
$\bu=u_r\be_r + u_\theta\be_\theta + u_z\be_z$ by using the cylindrical basis $\be_r$, $\be_\theta$, and $\be_z$.
Then, write $\bu=\bu^s+\bu^a$,
where $\bu^s$ and $\bu^a$ denote the symmetric (with respect to $\theta=0$) and antisymmetric parts of $\bu$ respectively \cite[page 252 figure 8.3]{DES:2011}, and consider its partial Fourier series decomposition:
\begin{equation} \label{vector_fourier}
\begin{aligned}
\bu^s &= 
\begin{pmatrix}
u_r^0 \\
0 \\
u_z^0
\end{pmatrix}+ 
\sum_{n=1}^\infty
\begin{pmatrix}
u_r^n \cos n\theta \\
u_\theta^n \sin n\theta \\
u_z^n \cos n\theta
\end{pmatrix}, \\
\bu^a &= 
\begin{pmatrix}
0 \\
u_\theta^0 \\
0
\end{pmatrix}+ 
\sum_{n=1}^\infty
\begin{pmatrix}
u_r^{-n} \sin n\theta \\
u_\theta^{-n} \cos n\theta \\
u_z^{-n} \sin n\theta
\end{pmatrix}.
\end{aligned}
\end{equation}

Next, consider the usual $\grad, \bcurl$, and $\dive$-operators in cylindrical coordinates:
\begin{equation*} %\label{grad_curl_div}
\begin{aligned}
\grad u &= (\partial_r u, \frac{1}{r}\partial_\theta u, \partial_z  u )^T, \\
\bcurl \bu &= (\frac{1}{r}\partial_\theta u_z - \partial_z u_\theta, \partial_z u_r-\partial_r u_z, \frac{1}{r}(\partial_r (ru_\theta)-\partial_\theta u_r))^T, \\
\dive \bu &= \frac{1}{r}\partial_r (ru_r) + \frac{1}{r}\partial_\theta u_\theta + \partial_zu_z.
\end{aligned}
\end{equation*}

If one applies these operators to (\ref{scalar_fourier}) or (\ref{vector_fourier}), then
by the orthogonality of $\cos n\theta$ and $\sin n\theta$, the individual $n$-th order Fourier modes in
$\grad u$, $\bcurl\bu$, and $\dive\bu$
decouple from one another in a weak formulation. The resulting $\grad, \bcurl,$ and $\dive$ formulas
that affect the $n$-th Fourier mode, and therefore the 
main operators of interest in Fourier-FEMs, are the following:
\begin{equation} \label{formulas}
\begin{aligned}
\grad\nolimits_{rz}\nolimits^{n} u 
&= \begin{bmatrix}
      \partial_r u\\
      -\frac{n}{r}u \\
       \partial_z u   
     \end{bmatrix}, \\
\bcurl\nolimits_{rz}\nolimits^{n} 
\begin{bmatrix}
       u_r \\
       u_\theta \\
       u_z           
     \end{bmatrix}
&= 
\begin{bmatrix}
       -(\frac{n}{r}u_z + \partial_z u_\theta)\\
       \partial_z u_r - \partial_r u_z \\
        \frac{nu_r+u_\theta}{r} + \partial_r u_\theta           
     \end{bmatrix}, \\
\dive\nolimits_{rz}\nolimits^{n} 
\begin{bmatrix}
       u_r \\
       u_\theta \\
       u_z           
     \end{bmatrix}
&= 
\partial_ru_r + \dfrac{u_r-nu_\theta}{r} + \partial_zu_z.
\end{aligned}
\end{equation}
Throughout this paper, the variable $n$ will be used to indicated the $n$-th Fourier mode of the function of interest.
Note that when $n=0$, the operators above become operators that arise in axisymmetric problems with axisymmetric data.
Since, we are extending the known results for $n=0$ to general data in this paper, we assume that $n>0$. 

Next, let $L^2(\breve{\Omega})$ denote the function space consisting of square integrable functions on $\breve{\Omega}$,
and let $\breve{L}^2(\breve{\Omega})$ denote the closed subspace of $L^2(\breve{\Omega})$ that consist of functions that are independent of the $\theta$-variable.
We then consider the $2D$ domain $\Omega\subset\RRR^2_+$ associated to $\breve{\Omega}$ and define the following weighted Hilbert space:
\[ 
L^2_r(\Omega) = \left\{u: \int\int_\Omega u(r,z)^2 rdrdz < \infty \right\}.
\]
Then, there is an isometry (up to a factor of $\sqrt{2\pi}$) between 
$\breve{L}^2(\breve{\Omega})$ and $L^2_r(\Omega)$,  since
\[
\int_{\breve{\Omega}} \breve{u}(r,\theta,z)^2 rdrd\theta dz = 2\pi\int_\Omega u(r,z)^2 rdrdz,
\]
where $u(r,z)\in L^2_r(\Omega)$ is the function that has the same formula as $\breve{u}(r,\theta,z)\in\breve{L}^2(\breve{\Omega})$. 
The inner-product and norm associated with $L^2_r(\Omega)$ will be denoted as follows:
\begin{align*}
(u,v)_{L^2_r(\Omega)} &= \int_\Omega uv rdrdz, \\
\| u \|_{L^2_r(\Omega)} &= \int_\Omega u^2 rdrdz.
\end{align*}

The weighted Hilbert spaces associated with the operators (\ref{formulas}) are summarized below.
\begin{align*} 
H_r( \grad\nolimits^n , \Omega) &= \left\{ u \in L^2_r(\Omega): \grad\nolimits_{rz}^n u \in L^2_r(\Omega) \times L^2_r(\Omega) \times L^2_r(\Omega)  \right\}, \\
\bH_r(\bcurl\nolimits^n, \Omega)  &=  \left\{ \bu\in L^2_r(\Omega) \times L^2_r(\Omega) \times L^2_r(\Omega):  
\bcurl\nolimits_{rz}\nolimits^{n}\bu \in L^2_r(\Omega) \times L^2_r(\Omega) \times L^2_r(\Omega) \right\}, \\
\bH_r(\dive\nolimits^n, \Omega) &= \left\{ \bu\in L^2_r(\Omega) \times L^2_r(\Omega) \times L^2_r(\Omega): \dive\nolimits_{rz}^n\bu \in L^2_r(\Omega) \right\}. 
\end{align*}
The associated inner-product to each space is
\begin{align*}
(u,v)_{H_r( \grad\nolimits^n , \Omega)} &= (u,v)_{L^2_r(\Omega)} + (\grad\nolimits_{rz}^n u, \grad\nolimits_{rz}^n v)_{L^2_r(\Omega)}, \\
(\bu,\bv)_{\bH_r(\bcurl\nolimits^n, \Omega) } &= (\bu,\bv)_{L^2_r(\Omega)} + (\bcurl\nolimits_{rz}\nolimits^{n}\bu, \bcurl\nolimits_{rz}\nolimits^{n}\bv)_{L^2_r(\Omega)}, \\
(\bu,\bv)_{\bH_r(\dive\nolimits^n, \Omega)} &=  (\bu,\bv)_{L^2_r(\Omega)} + (\dive\nolimits_{rz}^n\bu,  \dive\nolimits_{rz}^n\bv)_{L^2_r(\Omega)}. 
\end{align*}
%Throughout this paper, we will write $\|\cdot\|_r$ instead of $\| \cdot \|_{L^2_r(\Omega)}$ 
%and $(\cdot,\cdot)_r$ instead of $(\cdot,\cdot)_{L^2_r(\Omega)}$.

We will use $\Gamma_0$ (open) to denote the part of the boundary of $\Omega$ that is on the axis of symmetry (the $r=0$ axis), and 
$\Gamma_1$ to denote part of the boundary of $\Omega$ that is not on the axis of symmetry.
In other words, $\partial\Omega=\Gamma_0\cup\Gamma_1$, and a $2\pi$-rotation of $\Gamma_1$ around the axis of symmetry returns the entire boundary of $\breve{\Omega}$. 
We will use boldface when writing a vector-valued function or a function space that consists of vector-valued functions. 
For any vector $\bv$ of length-three, we will use $v_r$ to denote the $r$-component of $\bv$,
$v_\theta$ to denote the $\theta$-component of $\bv$, and so on, i.e., $\bv=v_r\be_r + v_\theta\be_\theta + v_z\be_z$.
Also, we will use $\bv_{rz}$ to denote $(v_r,v_z)$, the length-two vector that consists of the $r$ and $z$ components of $\bv$. 

\section{The Abstract Hodge Laplacian on Axisymmetric Domains and the Mixed Formulation} \label{section3}

In this section, we discuss the weighted Hodge Laplacian problem in a unified way following the framework of \cite{AFW:2010}.
We assume that $n>0$ is fixed and omit writing $n$ when writing the exterior derivative $d^k$, the co-derivative $\delta_{k}$,
and the function spaces $V^k$ and $V_k^*$ appearing in the domain complex and dual complex respectively, for $k=0,1,2,3$.
Throughout this paper, $k$ is an integer value between zero and three. 

Let $d^k$ be defined as the following
\begin{align*}
d^0 v &= \grad\nolimits^n\nolimits_{rz} v, \\
d^1 \bv &= \bcurl\nolimits^n\nolimits_{rz} \bv, \\
d^2 \bv &= \dive\nolimits^n\nolimits_{rz} \bv, \\
d^3 v &= 0,
\end{align*}
and let $V^k$ be the Hilbert space associated with it, i.e.,
\begin{equation} \label{Vk_space}
\begin{aligned}
V^0 &= H_r(\grad\nolimits^n, \Omega), \\
V^1 &= \bH_r(\bcurl\nolimits^n, \Omega), \\
V^2 &= \bH_r(\dive\nolimits^n, \Omega), \\
V^3 &=  L^2_r(\Omega).
\end{aligned}
\end{equation}
Viewing $d^k$ as an exterior derivative of differential forms that represents a Fourier series decomposition, 
one can calculate the dual operator of $d^k$ by calculating its co-derivative $\delta_k$. The detailed formulation of the domain complex and the dual complex 
using differential forms that represent our general axisymmetric setting can be found in subsection~\ref{differential_forms} in the Appendix.
Accordingly, we define $\delta_k$ in the following way:
\begin{align*}
\delta_0 &= 0, \\
\delta_1\bv &= -\dive\nolimits^{n*}_{rz}\bv, \\
\delta_2\bv &= \bcurl\nolimits^{n*}_{rz}\bv, \\
\delta_3v &=-\grad\nolimits_{rz}\nolimits^{n*}v,
\end{align*}
where 
\begin{equation} \label{dualoperator}
\begin{aligned}
\grad\nolimits_{rz}\nolimits^{n*} u 
&= \begin{bmatrix}
      \partial_r u\\
      \frac{n}{r}u \\
       \partial_z u   
     \end{bmatrix}, \\
\bcurl\nolimits_{rz}\nolimits^{n*} 
\begin{bmatrix}
       u_r \\
       u_\theta \\
       u_z           
     \end{bmatrix}
&= 
\begin{bmatrix}
       \frac{n}{r}u_z - \partial_z u_\theta\\
       \partial_z u_r - \partial_r u_z \\
        \frac{-nu_r+u_\theta}{r} + \partial_r u_\theta           
     \end{bmatrix}, \\
\dive\nolimits_{rz}\nolimits^{n*} 
\begin{bmatrix}
       u_r \\
       u_\theta \\
       u_z           
     \end{bmatrix}
&= 
\partial_ru_r + \dfrac{u_r+nu_\theta}{r} + \partial_zu_z.
\end{aligned}
\end{equation}
We will also need the following Hilbert spaces associated with these operators.
\begin{equation} \label{Vk_star}
\begin{aligned}
V_3^* &= H_{r,0}(\grad\nolimits^{n*}, \Omega) \\ 
&= \{u\in L^2_r(\Omega): \grad\nolimits_{rz}\nolimits^{n*} u \in L^2_r(\Omega) \times L^2_r(\Omega) \times L^2_r(\Omega), u=0 \textnormal{ on } \Gamma_1 \}, \\
V_2^* &= \bH_{r,0}(\bcurl\nolimits\nolimits^{n*} , \Omega) \\ 
&= \left\{ \bu\in L^2_r(\Omega) \times L^2_r(\Omega) \times L^2_r(\Omega): 
\bcurl\nolimits_{rz}\nolimits^{n*} \bu \in L^2_r(\Omega) \times L^2_r(\Omega) \times L^2_r(\Omega), \bu_{rz}\cdot\bt = 0 \textnormal{ and } u_\theta = 0 \textnormal{ on } \Gamma_1 \right\}, \\
V_1^* &= \bH_{r,0}(\dive\nolimits^{n*}, \Omega) \\  
&= \left\{ \bu\in L^2_r(\Omega) \times L^2_r(\Omega) \times L^2_r(\Omega): \dive\nolimits_{rz}^{n*}\bu \in L^2_r(\Omega),  \bu_{rz}\cdot\bn=0  \textnormal{ on } \Gamma_1  \right\}, \\
V_0^* &= L^2_r(\Omega),
\end{aligned}
\end{equation}
where the boundary conditions in (\ref{Vk_star}) are understood through trace operators as usual. (See \cite{Lacoste:2000} or subsection~\ref{differential_forms}.)
Then, $\delta_{k+1}: V_{k+1}^*\rightarrow V_{k}^*$ is the dual operator of $d^k: V^k\rightarrow V^{k+1}$
with respect to the $L^2_r$-inner product, and
\begin{equation} \label{adjoint2}
(d^ku,v)_{L^2_r(\Omega)} = (u, \delta_{k+1} v)_{L^2_r(\Omega)}  \quad \forall  u\in V^k \quad\forall v\in V_{k+1}^*.
\end{equation}
This is due to Theorem~\ref{adjoint} and (\ref{cos_sin}).

\begin{remark}
It is important to observe that $d^k$ is a closed, densely defined operator on $L^2_r(\Omega)$ (or  $\bL^2_r(\Omega)=L^2_r(\Omega) \times L^2_r(\Omega) \times L^2_r(\Omega)$) such that the range of $d^k$ is in $\bL^2_r(\Omega)$ (or  $L^2_r(\Omega)$),
and $d^{k+1} \circ d^k=0$. In other words, 
\begin{equation} \label{fred}
0\rightarrow L^2_r(\Omega) \xrightarrow{d^0} \bL^2_r(\Omega) \xrightarrow{d^1} \bL^2_r(\Omega) \xrightarrow{d^2} L^2_r(\Omega) \rightarrow 0,
\end{equation}
is a (unbounded) Hilbert complex. A
Hilbert complex is closed if the range of each $d^k$ is closed in $L^2_r(\Omega)$ (or $\bL^2_r(\Omega)$). 
Since the range of $d^k$ is finite codimensional in the null space of $d^{k+1}$ (See \cite{O:2015}), 
(\ref{fred}) is a Fredholm complex. 
Therefore, (\ref{fred}) is a closed Hilbert complex. 
\end{remark}

The domain complex associated with the closed Hilbert complex (\ref{fred}) is
\begin{equation} \label{domain_complex}
\begin{CD}
0 \rightarrow H_r(\grad^n, \Omega)   @> d^0 >> \bH_r(\bcurl\nolimits^n, \Omega)    @>d^1 >>  \bH_r(\dive^n, \Omega)    @>d^2 >>   L^{2}_{r}(\Omega) \rightarrow 0.
\end{CD}
\end{equation}
\textcolor{black}{We note here that (\ref{domain_complex}) for each Fourier mode has finite dimensional cohomology. In fact, 
it was proved in \cite[Theorem 2.1]{O:2015}  that (\ref{domain_complex}) form an exact sequence without further geometrical assumptions \cite[Remark 2.1]{O:2015}, 
so the cohomology vanishes.} 
%the harmonic forms (the orthogonal complement of range of $d^{k-1}$ in the null space of $d^k$)
%vanish without further geometrical assumptions. (See \cite[Remark 2.1]{O:2015}.) }

Associated with this domain complex, we have a dual complex of the following form: 
\begin{equation} \label{dual_complex}
\begin{CD}
0 \leftarrow L^2_r(\Omega)     @< \delta_1 << \bH_{r,0}(\dive^{n*}, \Omega)    @< \delta_2 <<   \bH_{r,0}(\bcurl^{n*}, \Omega)     @<\delta_3 <<   H_{r,0}(\grad^{n*}, \Omega) \leftarrow 0.
\end{CD}
\end{equation}

Let $W^k$ denote the spaces appearing in (\ref{fred}). Since (\ref{fred}) is a closed Hilbert complex, we immediately get the following Hodge decomposition and Poincare inequality. 

\begin{align*}
W^k &= range(d^{k-1}) \oplus [null(d^k)]^{\perp_W}, \\
V^k &= range(d^{k-1}) \oplus [null(d^k)]^{\perp_V}, \\
\|v\|_{V^k} &\leq C_P\|d^kv\|_{W^k} &&\forall v\in [null(d^k)]^{\perp_V},
\end{align*}
where $[null(d^k)]^{\perp_W}$ is the orthogonal complement of the null space of $d^k$ with respect to the $L^2_r$-inner product,
and $[null(d^k)]^{\perp_V}$ is the orthogonal complement of the null space of $d^k$ with respect to the $V^k$-inner product.

Now, let us define the abstract Hodge Laplacian as in \cite{AFW:2010} in the following way: 
\[
L_k = d^{k-1}\delta_{k}+\delta_{k+1}d^k.
\]
Then, the domain of $L_k$ denoted by $D_L^k$ is
\[
D_L^k = \{u\in V^k\cap V_k^*: d^ku\in V^*_{k+1} \textnormal{ and } \delta_ku\in V^{k-1} \}.
\]
If $u\in D_L^k$ solves the Hodge Laplacian problem $L_ku=f$, then $u\in D_L^k$ satisfies
\[
(d^ku, d^kv)_{L^2_r(\Omega)}  + (\delta_k u, \delta_k v)_{L^2_r(\Omega)}  = (f, v)_{L^2_r(\Omega)}  \quad \forall v \in V^k \cap V_k^*
\]
by (\ref{adjoint2}). This problem is well-posed, since the harmonic forms vanish.

Now, let us state the weighted mixed formulation of the abstract Hodge Laplacian $L_k$:

Find $(\sigma, u) \in V^{k-1}\times V^k$ such that
\begin{equation} \label{cts_mix}
\begin{aligned}
(\sigma,\tau)_{L^2_r(\Omega)}  - (d^{k-1}\tau, u)_{L^2_r(\Omega)}  &= 0,  &&\text{ for all } \tau\in V^{k-1}, \\
(d^{k-1}\sigma,v)_{L^2_r(\Omega)} + (d^ku,d^kv)_{L^2_r(\Omega)}  &= (f,v)_{L^2_r(\Omega)}, &&\text{ for all } v\in V^k.
\end{aligned}
\end{equation}

Since (\ref{fred}) is a closed Hilbert complex, the following theorem follows directly from \cite[Theorem 3.1]{AFW:2010}.
\begin{theorem}
The mixed formulation (\ref{cts_mix}) is well-posed, and for any $f\in L^2_r(\Omega)$, we have
\[
\| u \|_{V^k} + \| \sigma \|_{V^{k-1}} \leq \| f \|_{L^2_r(\Omega)}.
\]
\end{theorem}

It is straightforward to check that the four Hodge Laplacian problems arising from our axisymmetric setting correspond to the following problems. 
\begin{itemize}
\item $k=0$: The Neumann Problem for the Axisymmetric Poisson Equation 
%Find $u\in H_r(\gradrz^n,\Omega)$ such that
\begin{align*}
-\dive\nolimits^{n*}_{rz} \gradrz^n u &= f &&\textnormal{ in } \Omega, \\
\gradrz^n u \cdot \bn &= 0 &&\textnormal{ on } \Gamma_1.
\end{align*} 

\item $k=1$: The Axisymmetric Vector Laplacian $\bcurl\bcurl+\grad\mathrm{div}$
%Find $\bu=(u_r,u_\theta,u_z)\in\bH(\curlrz^n,\Omega)$ and $p\in\mathfrak{H}_{r, \diamond}^1$ such that
\begin{align*}
-\gradrz^n \dive\nolimits^{n*}_{rz}\bu + \bcurl\nolimits^{n*}\nolimits_{rz}\bcurl\nolimits^n\nolimits_{rz}\bu &= \bff, \\
(\bcurl\nolimits^n\nolimits_{rz}\bu)_{rz} \cdot \bt = 0, (\bcurl\nolimits^n\nolimits_{rz}\bu)_\theta &= 0 \textnormal{ on } \Gamma_1, \\
\bu_{rz} \cdot \bn &= 0 \textnormal{ on } \Gamma_1.
\end{align*}

\item $k=2$: The Axisymmetric Vector Laplacian $\bcurl\bcurl+\grad\mathrm{div}$
\begin{align*}
\bcurl\nolimits^n\nolimits_{rz}\bcurl\nolimits^{n*}\nolimits_{rz}\bu - \gradrz^{n*}\dive\nolimits^n_{rz}\bu &= \bff, \\
\bu_{rz}\cdot\bt = 0, u_\theta = 0, \dive\nolimits^n_{rz}\bu &= 0 \textnormal{ on } \Gamma_1.
\end{align*}

\item $k=3$: The Dirichlet Problem for the Axisymmetric Poisson equation
\begin{align*}
-\dive\nolimits^n_{rz}\gradrz^{n*} u &= f &&\textnormal{ in } \Omega, \\
u &= 0 &&\textnormal{ on } \Gamma_1.
\end{align*} 
\end{itemize}

\section{Approximation of Weighted Hilbert Complexes} \label{section4}

In this section, we will study the discrete weighted mixed formulation of the Hodge Laplacian. We will use the Fourier finite element spaces constructed in \cite{O:2015}. 
Let us first define the following polynomial spaces:
\begin{align*}
A_1 &= \left\{ \alpha_1r+\alpha_2r^2+\alpha_3rz: \alpha_i\in\RRR \text{ for } 1\leq i\leq 3 \right\}, \\
\bB_1 &= \left\{ 
\begin{pmatrix}
       \beta_1 + \beta_4 r + \beta_3 z - \beta_6 rz \\
       -n\beta_1 + \beta_2 r - n\beta_3 z \\
       \beta_5 r + \beta_6 r^2
     \end{pmatrix}: \beta_i\in\RRR \text{ for } 1\leq i \leq 6
\right\}, \\
\bC_1 &= \left\{ 
\begin{pmatrix}
       n\gamma_1 + \gamma_2 r \\
      \gamma_1 + \gamma_3 r \\
       \gamma_4 + \gamma_2 z
     \end{pmatrix}: \gamma_i\in\RRR \text{ for } 1\leq i \leq 4 
\right\}. 
\end{align*}
Assume that $\Omega$ is meshed by a finite element triangulation $\TT_h$ that satisfies the usual geometrical conformity conditions~\cite{fem:ciarlet}.
Then define the following global finite element spaces:
\begin{equation} \label{fem_spaces}
\begin{aligned}
A_h &= \left\{u\in H_r( \grad\nolimits^n , \Omega): u|_K\in A_1 \text{ for all } K\in \TT_h  \right\}, \\
\bB_h &= \left\{\bu\in \bH_r(\bcurl\nolimits^n , \Omega): \bu|_K\in \bB_1 \text{ for all } K\in \TT_h  \right\}, \\
\bC_h &= \left\{\bu\in \bH_r(\dive\nolimits^n , \Omega): \bu|_K\in \bC_1 \text{ for all } K\in \TT_h  \right\}, \\
D_h &= \left\{u\in L^2_r(\Omega): u|_K \text{ is constant for all } K\in \TT_h  \right\}.
\end{aligned}
\end{equation}

The main significance of this family of Fourier finite element spaces is that they make the following diagram commute:
\begin{equation} \label{disc_exact}
\begin{CD}
H_r(\grad^n, \Omega)   @>\mathbf{grad}^n_{rz}>> \bH_r(\bcurl\nolimits^n, \Omega)    @>\bcurl\nolimits^n\nolimits_{rz}>> \bH_r(\dive\nolimits^n, \Omega)     @>\dive^n_{rz}>>   L^{2}_{r}(\Omega) \\ %\cap L^2_{1/r}(\Omega)  \\
@VV I_g^n V                                    @VV I_c^n V                           @VV I_d^n V @ VV I_o V  \\
A_h @> \grad_{rz}^n >> \bB_h @> \bcurl\nolimits^n\nolimits_{rz} >> \bC_h @> \dive_{rz}^n >> D_h
\end{CD}
\end{equation}
where the interpolation operators used in the above commuting diagram are defined in the following way (See \cite[Section 3]{O:2015}):
\begin{align} 
\label{pi_grad}
I_g^n u |_K &= \sum_{i=1}^3 \Big( (\dfrac{n}{r}u)(\ba_i) \Big) \dfrac{r}{n}\lambda_i, \\
\label{pi_curl}
I_c^n \begin{pmatrix}
       u_r \\
       u_\theta \\
       u_z          
     \end{pmatrix} |_K &=
     \sum_{i=1}^3 
     u_\theta (\ba_i)
     \begin{pmatrix}
       -\dfrac{1}{n}\lambda_i \\
       \lambda_i \\
        0          
     \end{pmatrix}  +
     \sum_{i=1}^3
    ( \int_{e_i}
      \begin{pmatrix}
       \frac{nu_r + u_\theta}{r} \\
      \frac{nu_z}{r}     
       \end{pmatrix} \cdot \bt_i)
         \begin{pmatrix}
        \dfrac{r}{n}\nu_i^r \\
        0 \\
        \dfrac{r}{n}\nu_i^z    
       \end{pmatrix}, \\
       \label{pi_div}
       I_d^n \begin{pmatrix}
       u_r \\
       u_\theta \\
       u_z          
     \end{pmatrix} |_K &=
     (\dfrac{1}{|K|}\int_K \dfrac{nu_\theta - u_r}{r})
     \begin{pmatrix}
        0 \\
        \dfrac{r}{n}\chi_K \\
        0          
     \end{pmatrix}  +
     \sum_{i=1}^3
   ( \int_{e_i} 
     \begin{pmatrix}
       u_r \\
       u_z     
       \end{pmatrix} \cdot \bn_i ) 
         \begin{pmatrix}
        \xi_i^r \\
        \dfrac{1}{n}\xi_i^r \\
        \xi_i^z    
       \end{pmatrix}, \\
       \label{pi_l2}
       I_o u |_K &= (\dfrac{1}{|K|}\int_K u drdz) \chi_K.
\end{align}
In the above definition, we are using $\ba_i$ to denote the $i$-th vertex of triangle $K\in\TT_h$, and $\bt_i$ and $\bn_i$ are being used to denote the unit vector tangent and normal to the $i$-th edge of $K$ respectively. 
For each $\bx\in K$, $\lambda_i(\bx)$ denote its barycentric coordinates in $K$ so that $\bx=\sum_{i=1}^3 \lambda_i(\bx) \ba_i$, and
$\boldsymbol{\nu}_i= 
\begin{pmatrix}
\nu_i^r \\
\nu_i^z
\end{pmatrix}=
\lambda_j\nabla\lambda_k-\lambda_k\nabla\lambda_j$ for $(i,j,k)$-circular permutation notation, and
$\boldsymbol{\xi}_i=
\begin{pmatrix}
\xi_i^r \\
\xi_i^z
\end{pmatrix}$ is the rotation of $\boldsymbol{\nu}_i$ counter-clockwise by $\dfrac{\pi}{2}$, i.e.,
\begin{align} \label{xi_nu}
\begin{pmatrix}
\xi_i^r \\
\xi_i^z
\end{pmatrix}
&=
\begin{pmatrix}
-\nu_i^z \\
\nu_i^r 
\end{pmatrix}.
\end{align}
In other words, $\{\boldsymbol{\nu}_i\}$ form a local basis for the lowest order \Nedelec space (\cite{fem:nedelec}), and $\{\boldsymbol{\xi}_i\}$ form a local basis for the lowest order Raviart Thomas space (\cite{RT:1977}).
$\chi_K$ denotes the constant function one on triangle $K$. \textcolor{black}{These interpolation operators also satisfy error estimates which can be found in \cite[Theorem 4.1]{O:2015}. }

\begin{remark}
Note that the finite element space $\bB_h$ was constructed in \cite{Lacoste:2000} to approximate the solution of the axisymmetric time harmonic Maxwell equations, 
but in a different way compared to how it was constructed in \cite{O:2015}. 
\end{remark} 
  
Now let us focus on the discrete subcomplex appearing in (\ref{disc_exact}):
\begin{equation} \label{discrete_complex}
0\rightarrow A_h\xrightarrow{d^0} \bB_h \xrightarrow{d^1} \bC_h \xrightarrow{d^2} D_h \rightarrow 0.
\end{equation}
This is again a closed Hilbert complex, and we can get the weighted Hodge decomposition and the Poincare inequality as we did for the continuous closed Hilbert complex. 
Note that (\ref{discrete_complex}) is also exact. For a unified notation, we use $V_h^k$ to denote the discrete function spaces in (\ref{discrete_complex}), i.e.,
$V_h^0 = A_h, V_h^1 = \bB_h$, etc. 
We will use the notation $(V_h, d)$ to denote this Hilbert subcomplex (\ref{discrete_complex})
while we use $(V,d)$ to denote the Hilbert complex (\ref{domain_complex}).

Now, let us consider the following discrete weighted mixed formulation of the Hodge Laplacian:
find $(\sigma_h, u_h) \in V_h^{k-1}\times V_h^k$ such that
\begin{equation} \label{disc_mix}
\begin{aligned}
(\sigma_h,\tau_h)_{L^2_r(\Omega)}  - (d^{k-1}\tau_h, u_h)_{L^2_r(\Omega)}  &= 0,  &&\tn{ for all } \tau_h\in V_h^{k-1}, \\
(d^{k-1}\sigma_h,v_h)_{L^2_r(\Omega)} + (d^ku_h,d^kv_h)_{L^2_r(\Omega)}  &= (f,v_h)_{L^2_r(\Omega)} , &&\tn{ for all } v_h\in V_h^k.
\end{aligned}
\end{equation}

A uniformly $V$-bounded cochain projection from the complex $(V,d)$
to the subcomplex $(V_h,d)$ is a set of operators $\Pi_h^k: V^k \rightarrow V_h^k$, $k=0,1,2,3$, such that
\begin{equation} \label{bounded}
\begin{aligned}
\Pi_h^k v_h &= v_h &&\textnormal{ for all } v_h\in V_h^k, \\ 
d_k\Pi_h^k &=\Pi_h^{k+1}d_k, \\
\| \Pi_h^kv \|_V &\leq C\| v \|_V &&\textnormal{ for all } v\in V^k,
\end{aligned}
\end{equation}
for some constant $C$ independent of the mesh parameter $h$.
It was established in \cite[Theorem 3.9]{AFW:2010} that a uniformly $V$-bounded cochain projection guarantees the stability of $(\ref{disc_mix})$
along with error estimates between $(\sigma,u)$ of (\ref{cts_mix}) and $(\sigma_h,u_h)$ of (\ref{disc_mix}). 

With this goal in mind, we construct a uniformly $V$-bounded cochain projection from the complex $(V,d)$
to the subcomplex $(V_h,d)$ in the next section. Then, by \cite[Theorems 3.8 and 3.9]{AFW:2010}, we achieve the main result of this paper, the stability result along with abstract error estimates. 

\begin{theorem} \label{main2}
The discrete mixed formulation (\ref{disc_mix}) is stable. Furthermore, if 
$(\sigma,u)\in V^{k-1}\times V^k$ is the solution to (\ref{cts_mix}) and
$(\sigma_h,u_h)\in V^{k-1}_h\times V^k_h$ is the solution to (\ref{disc_mix}), then 
\[
\|\sigma-\sigma_h\|_{V^{k-1}} + \| u-u_h \|_{V^k} \leq C(\inf_{\tau_h\in V_h^{k-1}}\|\sigma-\tau_h\|_{V^{k-1}} + \inf_{v_h\in V_h^{k}}\|u-v_h\|_{V^{k}}). 
\]
\end{theorem}

\begin{remark}
As usual, the interpolation operators in (\ref{disc_exact}) do not form a $V$-bounded cochain projection,
since it is not well-defined for all functions in $V^k$ but only for functions in $V^k$ with extra regularity.
\end{remark}

\section{A Bounded Cochain Projection for Fourier-FEMs} \label{section5}

In this section, we construct uniformly $V$-bounded cochain projections that satisfy (\ref{bounded}). In fact, the projections constructed here 
are $W$-boounded cochain projections which can be applied to not only all functions in $V^k$ but also in $W^k$. 
In \cite{GO:2012}, the well-known work of Schoberl (\cite{Schob01, Schberl2005NumericalMF}) and Christiansen and Winther (\cite{Christiansen:2008}) 
were modified to weighted function spaces, and we constructed $W$-bounded cochain projections that can be applied to axisymmetric problems with axisymmetric data. 
We now extend the work of \cite{GO:2012} to construct $W$-bounded cochain projections that can be applied to axisymmetric problems with general data.
\textcolor{black}{Before we discuss the detailed construction of these new commuting smoothed projections, 
we first point out the differences between these projections versus the ones already constructed in \cite{GO:2012}.
% and some new ideas used in their construction.
The projections constructed in \cite{GO:2012} satisfy the commuting diagram property with operators 
\begin{equation} \label{easy}
\begin{aligned}
\gradrz u &= (\partial_r u, \partial_z u)^T, \\
\curlrz (v_r,v_z) &= \partial_zv_r - \partial_rv_z,
\end{aligned}
\end{equation}
while the new projections must satisfy the commuting diagram property with operators in (\ref{formulas}). 
The formulas in $(\ref{easy})$ are significantly simpler than those in (\ref{formulas}), and they are also similar to the standard 
gradient and curl operators in $2D$.  
Operators in (\ref{formulas}) are more complicated and quite different from the standard ones, so special attention is required when constructing smoothed projections to ensure commutativity with these operators.
Furthermore, the multiple $\dfrac{n}{r}$-terms appearing in (\ref{formulas}) cause difficulties in the analysis that one does not have to deal with in \cite{GO:2012}. One of the new ideas used to overcome this difficulty is 
in the use of $r^3$ (instead of $r$) when defining $\eta_a$ in Proposition~\ref{prop:eta} and $\kappa_i$ in (\ref{eq:kappa}). This is explained in detail in Remark~\ref{good}.
}

Let us start by constructing smoothing operators that satisfy the commuting diagram property.
Let $K$ be a triangle in $\TT_h$,
$r(\by)$ denote the value of the radial coordinate
at a point $\by \in \RRR^2_{+}$,
\[
h_K=\diam(K), 
\qquad 
r_K = \max_{\bx \in K}r(\bx),
\]
and $\rho_K$ denote the diameter of the largest circle inscribed in $K$.  
We assume that 
\[
\dfrac{h_K}{\rho_K}<C,
\] 
i.e., the triangular mesh $\TT_h$ is shape regular.
Finally, let $P_l$ denote the space of polynomials of degree at most~$l$ (for
some $l\ge 0$). Throughout this paper, we use $C$ to denote a
generic positive constant that is independent of $\{h_K\}$. 
%Its value may differ at different occurrences and can depend on shape regularity ratio $h_K/\rho_K$, but not on $h_K$ itself. 
%We will define another gradient operator denoted by $\gradrz$ (without any superscript):
%\[
%\gradrz v_h = (\partial_r v_h, \partial_z v_h)^T.
%\]

The first three items in the following proposition can be found in \cite[page 4]{GO:2012} while the fourth item is an extra result that is needed to extend the results in that paper to general axisymmetric problems. 
Its proof can be found in the Appendix (subsection~\ref{inverse_proof}). 

\begin{prop} 
  \label{prop:inverse}
  For all $v\in P_l$,
  \begin{align}
  \label{inv1}
    \| \gradrz v \|^2_{L_r^2(K)}
    &\leq C h_K^{-2} \| v \|^2_{L_r^2(K)}, \\
    \label{inv2}
    \| \gradrz v \|^2_{L^\infty(K)}
    &\leq C h_K^{-2} \| v \|^2_{L^\infty(K)}, \\
    \label{inv3}
    r_K h_K^2 \| v \|^2_{L^\infty(K)} 
     &\leq C \| v \|^2_{L_r^2(K)},
     \end{align}
    Furthermore, if $u\in P_l$ vanishes on $\Gamma_0$, then we have
    \begin{align} \label{inv4} 
     r_K^2 \left\| \frac{u}{r} \right\|^2_{L^2_r(K)} &\leq C\left\| u \right\|_{L^2_r(K)}^2.
     \end{align}
\end{prop}

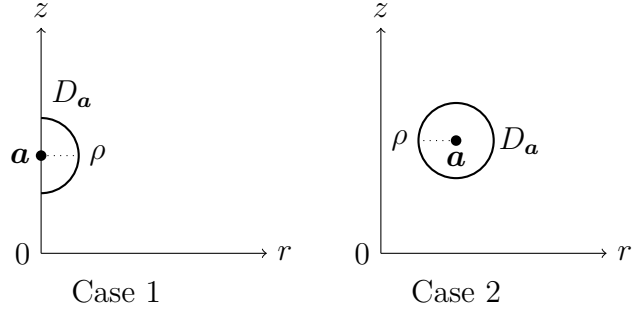
\begin{figure} \label{figfig21}
  \centering
  \begin{tikzpicture}
    \draw[->] (0,0) -- (0,3)  node [anchor=south] {$z$};
    \draw[->] (0,0) -- (3,0)  node [anchor=west] {$r$};
    \coordinate [label=left:$0$]   (o) at (0,0);
    \coordinate [label=left:$\ba$] (a) at (0,1.3);
    \fill (a) circle (2pt);     
    \coordinate (ap) at ($(a)-(0,0.5)$);
   \draw[thick] (ap) arc (-90:90:0.5cm) node [anchor=south]  {\qquad$D_\ba$};
    \draw [dotted] (a)--($(a)+(0.5,0)$) node [anchor=west] {$\rho$};
    \node at (1,-.5) {Case~1};
  \end{tikzpicture}
  \quad
  \begin{tikzpicture}
    \draw[->] (0,0) -- (0,3)  node [anchor=south] {$z$};
    \draw[->] (0,0) -- (3,0)  node [anchor=west] {$r$};
    \coordinate [label=left:$0$]   (o) at (0,0);
    \coordinate [label=below:$\ba$] (a) at (1,1.5);
    \fill (a) circle (2pt);     
   \draw[thick] (a) circle (0.5cm) node [anchor=west] {\quad$D_\ba$};
    \draw [dotted] (a)--($(a)-(0.5,0)$) node [anchor=east] {$\rho$};
    \node at (1,-.5) {Case~2};
  \end{tikzpicture}
    \caption{Domains $D_\ba$ corresponding to point $\ba$.}
    \label{figfig21}
\end{figure}

Let $\ba = (a_r, a_z)$ be a point in $\RRR^2_{+}$ and let $\Da$ be a
closed disk of radius $\rho$, or its half, centered around $\ba$, as shown in 
Figure~\ref{figfig21}. Using this notation, we have the following proposition, and its proof can be found in the Appendix (subsection~\ref{prop_proof}).  

\begin{prop}\label{prop:eta}
There exists a function $\eta_a(r,z)\in P_l$, for any $l\geq 0$, such that
\begin{align}
\label{eta1} 
(r\eta_a,rp)_{L^2_r(D_\ba)} &= p(\ba), &&\textnormal{ for all } p\in P_l, \\
\label{eta2}
\left\| r\eta_a \right\|^2_{L^2_r(D_\ba)} &\leq \dfrac{C}{\rho^2 r_a^3}, &&\textnormal{ where } r_{\ba} = \left\{
     \begin{array}{ll}
       \rho,  &  \textnormal{ in Case 1, }  \\
       \displaystyle{\min_{\by \in \Da} \;r(\by)},  
       &  \textnormal{ in Case 2,}
     \end{array}
     \right. \\
\label{eta3}
\int_{D_\ba} r^3(\by)|\eta_{\ba}(\by)| d\boldsymbol{y} &\leq C.
\end{align}
\end{prop}
Note that (\ref{eta1}) implies that
\begin{equation} \label{eta4}
\int_{D_\ba} r^3(\by)\eta_\ba(\by) d\by = 1.
\end{equation}
Next, we will describe how to choose a $D_\ba$ for each mesh vertex $\ba$ in $\TT_h$. 
Let $\delta>0$ be a global parameter. %The criteria in selecting $\delta$ will follow shorty.
\begin{itemize}
\item If $\ba$ is on the $z$-axis, then $D_\ba$ is set as in Case~1 of Figure~\ref{figfig21} with $\rho=h\delta$.
\item If $\ba$ is not on the $z$-axis, then $D_\ba$ is set as in Case~2 of Figure~\ref{figfig21} with $\rho=h\delta$. 
%\item If $\ba$ is on $\Gamma_1$, then $D_\ba^h$ is set to $D_{\tilde{\ba}}$ as in Case~2 of Figure~\ref{fig:domains} with $\rho=h\delta$, where $\tilde{\ba}$ denotes the shifted $\ba$ so that it satisfies the conditions listed below.
\end{itemize}
Next, let $\Omega_\ba$ denote the vertex patch of the mesh vertex $\ba$. We choose $\delta>0$ small enough that the following conditions are satisfied for all $D_\ba$. 
\begin{enumerate} \label{numbers}
\item $D_\ba\subset\Omega_\ba$ for all vertices $\ba$ that are not on $\Gamma_1$.
\item $D_\ba\cap \Omega \subset\Omega_\ba$ for all vertices $\ba$ that are on $\Gamma_1$.
\item $D_\ba$'s of different mesh vertices do not overlap.
\item\label{four} $r_a \geq \delta h$
\end{enumerate}
Note that, for $\ba$'s that lie on $\Gamma_1$, $D_\ba$ will not be fully in $\Omega$ but in an extension of $\Omega$ denoted by $\widetilde{\Omega}$. 
Let $V^k(\Omega)$ be the same as $V^k$ defined in (\ref{Vk_space}), and $V^k(\widetilde{\Omega})$ be defined in a similar way but on $\widetilde{\Omega}$ instead of $\Omega$.
Throughout this paper, we will assume that there exists a set of extension operators $E^k:V^k(\Omega)\rightarrow V^k(\widetilde{\Omega})$ for all $k=0,1,2,3$ that satisfy the following properties:
\begin{equation} \label{assumption}
\begin{aligned}
E^ku(r,z) &= u(r,z) \textnormal{ for all } (r,z)\in\Omega, \\ 
d^k \circ E^k &= E^{k+1} \circ d^k, \\
\|E^ku\|_{L^2_r(\widetilde{\Omega})} &\leq C\| u \|_{L^2_r(\Omega)}.
\end{aligned}
\end{equation}
For the rest of the paper, in the case when $\ba$ of interest is on $\Gamma_1$ and thus $D_\ba\subset\widetilde{\Omega}$, it is assumed that the extension $E^ku$ is used in place of $u$ when it is being evaluated at $D_\ba$.
The assumptions (\ref{assumption}) assures that Proposition~\ref{prop_comm} and Lemma~\ref{lem:Rbounds} remain true as it is stated even with such use of $E_k$. 

\begin{remark}
Since $\ba\in\Gamma_1$ is on the natural boundary, we may construct $E_k$ by modifying \cite[page 65]{AFW:2006} that uses a Lipschitz continuous bijection to the axisymmetric setting discussed in subsection~\ref{differential_forms}.
%In general, since $\ba\in\Gamma_1$ is on the natural boundary, we may choose $D_\ba^h\subset\tilde{\Omega}$ such that $E_ku|_{D_\ba^h}$ depends on $u|_{\Omega_a}$ only. 
%For vertices on the intersection of $\Gamma_0$ and $\Gamma_1$, we use a half circle as in case 1 and use the extension operator. %Details are included in the appendix. 
\end{remark}

%Therefore, we need an extension operator to talk about functions in $D_\ba^h$ for $\ba$'s on $\Gamma_1$. For this extension operator, we use the one used in [AFW:2006, page 65] and [Schoberl 2005], where they construct a bounded extension operator that commutes with the exterior derivative by using a Lipschitz bijection. 
%In general, since $\ba\in\Gamma_1$ is on the natural boundary, we may choose $D_\ba^h\subset\tilde{\Omega}$ such that $\tilde{u}|_{D_\ba^h}$ depends on $u|_{\Omega_a}$ only. 
%For vertices on the intersection of $\Gamma_0$ and $\Gamma_1$, we use a half circle as in case 1 and use the extension operator. 

We need some more notations to define the smoothing operators.  We use
$[\cdot, \cdot, \cdot]$ to denote the convex hull of its arguments.  Accordingly, a
triangle $K\in\TT_h$ with vertices $\ba_1$, $\ba_2$, and $\ba_3$ is $K=[\ba_1,
\ba_2, \ba_3]$.  Its three edges are $\fe_1 = [\ba_2, \ba_3]$, $\fe_2
= [\ba_3, \ba_1]$, and $\fe_3 = [\ba_1, \ba_2]$. Let $ D_{\ba_i}$ be
the smoothing domains introduced above for the vertex~$\ba_i$
$(i=1,2,3$), and let $\by_i\in D_{\ba_i}$.  
Set
\begin{equation}
  \label{eq:kappa}
  \kappa_i(\by_i) = r^3(\by_i)\eta_{\ba_i}(\by_i),
\end{equation}
where $\eta_{\ba_i}$ is the function given by 
Proposition~\ref{prop:eta} with $\ba$ equal to $\ba_i$.  
We write $\kappa_{123} =
\kappa_1\kappa_2\kappa_3$ and $\kappa_{12} = \kappa_1\kappa_2$, etc.
When it is more convenient to use a $(i,j,k)$-circular permutation notation, we will use $i, j, k$
instead of $1, 2, 3$ above, i.e., $K=[\ba_i, \ba_j, \ba_k]$, $\fe_i=[\ba_j, \ba_k]$, etc.

Recall that, for each $\bx\in K$, $\lambda_i(\bx)$ denote its barycentric
coordinates in~$K$ so that 
\[
\bx = \sum_{i=1}^3 \lambda_i(\bx)\ba_i.
\]
Following~\cite{Schob01}, we now define~$\xt$ by
\begin{equation} \label{xt}
 \xt(\bx, \by_1, \by_2, \by_3) = \sum_{i=1}^3 \lambda_i(\bx)\by_i
\end{equation}
and introduce these mesh dependent smoothers:

\begin{definition}
\begin{align}
\label{def:Sg}      S^gu(\bx) &= \dfrac{r}{n}\triint \kappa_{123} \;\left(\dfrac{nu}{r}\right)(\xt)\; \trid \\
\label{def:Sc}         S^c\bv(\bx) &= 
      \begin{bmatrix}
      \Bigg[\dfrac{r}{n}\triint \kappa_{123} 
      \left(\dfrac{\bd\xt}{\bd\bx}\right)^T 
      \begin{bmatrix} 
       \dfrac{nv_r+v_\theta}{r}(\xt) \\ \\ 
       \dfrac{nv_z}{r}(\xt) 
      \end{bmatrix}\; \trid\Bigg]_{r}   \\  -\dfrac{1}{n}\triint\kappa_{123} v_\theta(\xt)\trid  \\ \\ \\ \\
      \triint\kappa_{123} v_\theta(\xt)\trid \\ \\ \\ \\
       \Bigg[\dfrac{r}{n}\triint \kappa_{123} 
      \left(\dfrac{\bd\xt}{\bd\bx}\right)^T 
      \begin{bmatrix} 
       \dfrac{nv_r+v_\theta}{r}(\xt) \\ \\
          \dfrac{nv_z}{r}(\xt) 
      \end{bmatrix}\; \trid \Bigg]_{z} 
      \end{bmatrix} \\ 
\label{def:Sd}          S^d\bw(\bx) &=
      \begin{bmatrix}
      \Bigg[\triint \kappa_{123} \dett\left(\dfrac{\bd\xt}{\bd\bx}\right) \left(\dfrac{\bd\xt}{\bd\bx}\right)^{-1}
      \begin{bmatrix} 
      w_r(\xt) \\ 
       w_z(\xt) 
      \end{bmatrix} \trid\Bigg]_{r}  \\ \\ \\
      \dfrac{1}{n}\Bigg[\triint \kappa_{123} \dett\left(\dfrac{\bd\xt}{\bd\bx}\right) \left(\dfrac{\bd\xt}{\bd\bx}\right)^{-1}
      \begin{bmatrix} 
       w_r(\xt) \\ 
       w_z(\xt) 
      \end{bmatrix} \trid\Bigg]_{r}  \\ + \dfrac{r}{n}\triint \kappa_{123} \dett\left(\dfrac{\bd\xt}{\bd\bx}\right) \left(\dfrac{nw_\theta-w_r}{r}\right)(\xt) \trid \\ \\ \\
        \Bigg[\triint \kappa_{123} \dett(\dfrac{\bd\xt}{\bd\bx})(\dfrac{\bd\xt}{\bd\bx})^{-1}
      \begin{bmatrix} 
       w_r(\xt) \\ 
       w_z(\xt) 
      \end{bmatrix} \trid\Bigg]_{z} 
      \end{bmatrix} \\ 
\label{def:So}     S^ou(\bx) &= \triint \kappa_{123} \;\dett\left(\dfrac{\bd\xt}{\bd\bx}\right) u(\xt) \trid
\end{align}
  \end{definition}

\begin{prop} \label{prop_comm}
The smoothing operators satisfy a commuting diagram property:
\begin{align}
\label{com1} \grad\nolimits_{rz}\nolimits^n(S^gu) &= S^c(\grad\nolimits_{rz}\nolimits^nu), \\
\label{com2} \bcurl\nolimits_{rz}^n(S^c\bu) &= S^d(\bcurl\nolimits_{rz}^n\bu), \\
\label{com3} \dive\nolimits_{rz}\nolimits^n(S^d\bu) &= S^o(\dive\nolimits_{rz}\nolimits^n\bu).
\end{align}
\end{prop}

\begin{proof}
The results follow by construction, so we will only include the proof of (\ref{com1}) here. 
By the definition of $S^g$, $S^c$, and $\grad\nolimits_{rz}\nolimits^n$, we have
\begin{equation} \label{cp1}
\begin{aligned}
&S^c(\grad\nolimits_{rz}\nolimits^n u) = S^c ((\partial_r u, -\frac{nu}{r}, \partial_z u)^T) \\
&= \begin{bmatrix}
     \footnotesize \Bigg[\dfrac{r}{n} \triint \kappa_{123} \left(\dfrac{\bd\xt}{\bd\bx}\right)^T 
        \begin{bmatrix} 
      \frac{\partial_r (nu)}{r}-\frac{nu}{r^2} \\
        \frac{\partial_z(nu)}{r}
      \end{bmatrix}(\xt)  \trid\Bigg]_{r} + \dfrac{1}{n} \triint \kappa_{123} (\dfrac{nu}{r})(\xt) \trid   \\ \\
     \footnotesize -\triint \kappa_{123} (\dfrac{nu}{r})(\xt) \trid \\ \\
     \footnotesize  \Bigg[\dfrac{r}{n} \triint \kappa_{123} \left(\dfrac{\bd\xt}{\bd\bx}\right)^T 
     \begin{bmatrix} 
      \frac{\partial_r (nu)}{r}-\frac{nu}{r^2} \\
        \frac{\partial_z(nu)}{r}
      \end{bmatrix}(\xt) \trid\Bigg]_{z}
      \end{bmatrix}  \\
&=      \begin{bmatrix}
     \footnotesize \Bigg[\dfrac{r}{n} \triint \kappa_{123} \left(\dfrac{\bd\xt}{\bd\bx}\right)^T 
        \begin{bmatrix} 
      \partial_r (\frac{nu}{r}) \\
       \partial_z (\frac{nu}{r})
      \end{bmatrix}(\xt)  \trid\Bigg]_{r} + \dfrac{1}{n} \triint \kappa_{123} (\dfrac{nu}{r})(\xt) \trid   \\ \\
     \footnotesize -\triint \kappa_{123} (\dfrac{nu}{r})(\xt) \trid \\ \\
     \footnotesize  \Bigg[\dfrac{r}{n} \triint \kappa_{123} \left(\dfrac{\bd\xt}{\bd\bx}\right)^T 
     \begin{bmatrix} 
      \partial_r (\frac{nu}{r}) \\
       \partial_z (\frac{nu}{r})
      \end{bmatrix}(\xt) \trid\Bigg]_{z}
      \end{bmatrix}
\end{aligned}
 \end{equation}
 and
 \begin{equation} \label{cp2}
 \begin{aligned}
 \grad\nolimits_{rz}\nolimits^n (S^g u)
 &= \grad\nolimits_{rz}\nolimits^n\big(\dfrac{r}{n}\triint \kappa_{123} \;\left(\dfrac{nu}{r}\right)(\xt)\; \trid \big) \\
 &=  \begin{bmatrix}
     \footnotesize \partial_r \Bigg[\dfrac{r}{n}\triint \kappa_{123} \;\left(\dfrac{nu}{r}\right)(\xt)\; \trid\Bigg] \\ \\
      \footnotesize -\triint \kappa_{123} (\dfrac{nu}{r})(\xt) \trid \\ \\
       \footnotesize \partial_z \Bigg[\dfrac{r}{n}\triint \kappa_{123} \;\left(\dfrac{nu}{r}\right)(\xt)\; \trid\Bigg] 
      \end{bmatrix} \\
&=  \begin{bmatrix}
     \footnotesize \dfrac{r}{n} \partial_r \Bigg[\triint \kappa_{123} \;\left(\dfrac{nu}{r}\right)(\xt)\; \trid\Bigg] + \dfrac{1}{n}\triint \kappa_{123} \;\left(\dfrac{nu}{r}\right)(\xt)\; \trid \\ \\
      \footnotesize -\triint \kappa_{123} (\dfrac{nu}{r})(\xt) \trid \\ \\
       \footnotesize \dfrac{r}{n}\partial_z \Bigg[\triint \kappa_{123} \;\left(\dfrac{nu}{r}\right)(\xt)\; \trid\Bigg] 
      \end{bmatrix}.
 \end{aligned}
 \end{equation}
Since      
\[
\left(\dfrac{\bd\xt}{\bd\bx}\right)^T\Big(\grad\nolimits_{rz} \big(\dfrac{nu}{r}\big)\Big)(\xt) = \grad\nolimits_{rz}\Big(\big(\dfrac{nu}{r})(\xt)\Big),
\]
it follows that (\ref{cp1}) and (\ref{cp2}) are equal, and this completes the proof of (\ref{com1}).
The remaining results are proved in a similar way by using chain rule and the covariant transformation and the Piola transformation. (See \cite[Lemma 14]{Schberl2005NumericalMF}.)
\end{proof}

\begin{definition} 
\begin{align}
\label{Rhg} R_h^gu &= I_g^nS^gu, \\
\label{Rhc} R_h^c\bu &= I_c^nS^c\bu, \\
\label{Rhd} R_h^d\bu &= I_d^nS^d\bu, \\
\label{Rho} R_h^ou &= I_oS^ou.
\end{align}
\end{definition}

\begin{lemma}
  \label{lem:Rbounds}
There exists a constant $C$ independent of $h$ and $\delta$ such that
\begin{align}
  \label{eq:g1}
  \left\| R^g_hu \right\|^2_{L^2_r(\Omega)}
  & \leq \frac{C}{\delta^5} \left\| u \right\|^2_{L^2_r(\Omega)},
  &&\forall\; u \in L^2_r(\Omega),
  \\
  \label{eq:g2}
  \left\| R^g_hu_h - u_h \right\|_{L^2_r(\Omega)}^2 
  & \leq C\delta^2\left\| u_h \right\|_{L^2_r(\Omega)}^2,
  && \forall\;
  u_h \in A_h,
  \\ 
  \label{eq:c1}
  \left\| R^c_h\bv \right\|^2_{L^2_r(\Omega)}
  & \leq \frac{C}{\delta^5}\left\| \bv \right\|^2_{L^2_r(\Omega)},
  && \forall\; \bv \in L^2_r(\Omega)\times L^2_r(\Omega)\times L^2_r(\Omega),
  \\
  \label{eq:c2}
  \left\| R^c_h\bv_h - \bv_h \right\|_{L^2_r(\Omega)}^2
  & \leq C\delta^2\left\| \bv_h \right\|_{L^2_r(\Omega)}^2,
  &&\forall\; \bv_h \in \bB_h,
  \\
    \label{eq:d1}
  \left\| R^d_h\bv \right\|^2_{L^2_r(\Omega)}
  & \leq \frac{C}{\delta^3}\left\| \bv \right\|^2_{L^2_r(\Omega)},
  && \forall\; \bv \in L^2_r(\Omega)\times L^2_r(\Omega)\times L^2_r(\Omega),
  \\
  \label{eq:d2}
  \left\| R^d_h\bv_h - \bv_h \right\|_{L^2_r(\Omega)}^2
  & \leq C\delta^2\left\| \bv_h \right\|_{L^2_r(\Omega)}^2,
  &&\forall\; \bv_h \in \bC_h,
  \\
  \label{eq:o1}
  \left\| R^o_hw \right\|^2_{L^2_r(\Omega)}
  & \leq \frac{C}{\delta^3}\left\| w \right\|^2_{L^2_r(\Omega)},
  &&\forall \; w \in L^2_r(\Omega),
  \\
  \label{eq:o2}
  \left\| R^o_hw_h - w_h \right\|_{L^2_r(\Omega)}^2
  & \leq C\delta^2\left\| w_h \right\|_{L^2_r(\Omega)}^2,
  && \forall\;  w_h \in D_h.
\end{align}
\end{lemma}

\begin{proof}

We will only prove (\ref{eq:g1}), (\ref{eq:g2}), (\ref{eq:d1}), and (\ref{eq:d2}) in detail here
by modifying the the proof of \cite[Lemma 4.2]{GO:2012}.
In particular, we will combine the techniques used to prove multiple items in \cite[Lemma 4.2]{GO:2012}.
Items (\ref{eq:c1}), (\ref{eq:c2}), (\ref{eq:o1}), and (\ref{eq:o2}) 
can be proved in a similar way.
For simplicity, we assume that $n=1$ throughout this proof. 

Let $K=[\ba_1,\ba_2,\ba_3]$ be a fixed triangle in $\TT_h$.
We first we prove (\ref{eq:g1}) and (\ref{eq:g2}).
By (\ref{Rhg}) and (\ref{pi_grad}), we have
\[
R_h^gu |_{K} = \sum_{i=1}^3 \Bigg[(\dfrac{S^gu}{r})(\ba_i)\Bigg] r\lambda_i.
\] 
Therefore, 
\begin{equation} \label{one}
\begin{aligned}
\Lrnk{R_h^gu}^2 &= \Lrnk{\sum_{i=1}^3 \Bigg[\dfrac{S^gu}{r}(\ba_i)\Bigg] r\lambda_i}^2 \\
&\leq \sum_{i=1}^3 |\dfrac{S^gu}{r}(\ba_i)|^2 \int_K (r\lambda_i)^2 rdrdz \\
&\leq C\sum_{i=1}^3 |\dfrac{S^gu}{r}(\ba_i)|^2 r_K^3 h^2  &&\textnormal{ since } |\lambda_i|\leq 1 \textnormal{ for all } i=1, 2, 3.
\end{aligned}
\end{equation}
Let us first bound $|\dfrac{S^gu}{r}(\ba_i)|$. By definition, we have
\begin{equation} \label{xta}
\xt(\ba_i,\by_1,\by_2, \by_3)=\by_i \textnormal{ for all } i=1, 2, 3,
\end{equation}
and so by (\ref{def:Sg}), we have
\begin{equation} \label{two}
\begin{aligned}
|\dfrac{S^gu}{r}(\ba_i)| 
&= |\triint \kappa_{123} \; \dfrac{u}{r} (\by_i)\; \trid| &&\textnormal{ by (\ref{xta}) } \\
&= |\int_{D_\bai} \kappa_i(\by_i) \dfrac{u}{r}(\byi) d\byi| &&\tn{ by (\ref{eta4}) } \\
&= |\int_{D_\bai} (r\eta_\bai)(\byi) u(\byi) r(\byi) d\byi| &&\tn{ since } \kappa_i=r^3\eta_\bai \\
&\leq \|r\eta_\bai\|_{L_r^2(D_\bai)} \|u\|_{L_r^2(D_\bai)} \\
&\leq \dfrac{C}{\sqrt{\rho^2r_\bai^3}}  \|u\|_{L_r^2(D_\bai)} &&\tn{ by Propostion~\ref{prop:eta} item~(\ref{eta2}). }
\end{aligned}
\end{equation}
Therefore, continuing from (\ref{one}), we get
\begin{equation} \label{local}
\begin{aligned}
\Lrnk{R_h^gu}^2 &\leq \sum_{i=1}^3 \dfrac{C}{\rho^2r_\bai^3}  \|u\|^2_{L_r^2(D_\bai)} r_K^3 h^2 &&\tn{ by (\ref{two}) } \\
&\leq \dfrac{C}{\delta^2}\|u\|^2_{L_r^2(D_K)}\sum_{i=1}^3(\dfrac{r_K}{r_\bai})^3 &&\tn{ since $\rho=h\delta$ } \\
&\leq \dfrac{C}{\delta^5}\|u\|^2_{L_r^2(D_K)},
\end{aligned}
\end{equation}
where $D_K$ is the union of $D_{\ba_1}, D_{\ba_2}$, and $D_{\ba_3}$.
The last inequality above follows since $r_K\leq r_\bai + Ch$, so we have
\begin{equation}  \label{rkra}
\dfrac{r_K}{r_\bai} \leq \dfrac{r_\bai + Ch}{r_\bai} = 1 +  \dfrac{Ch}{r_\bai} \leq 1+\dfrac{C}{\delta} \leq \dfrac{C}{\delta},
\end{equation}
where we are using item~(\ref{four}) of the selection criteria of the smoothing domains.
By summing up over all triangles, and by using the shape regularity of $\TT_h$, (\ref{local}) implies (\ref{eq:g1}).

Now let $u_h$ be any function in $A_h$. 
We first find a bound for the degrees of freedom of $R_h^gu_h-u_h$:
\begin{equation} \label{dofdof1}
\begin{aligned}
|\dfrac{S^gu_h}{r}(\bai) - (\dfrac{u_h}{r})(\bai)| &= |\int_{D_\bai} \kappa_i (\dfrac{u_h}{r}(\byi) - \dfrac{u_h}{r}(\bai)) d\byi| \\
&\leq  \mathrm{max}_{\byi\in D_\bai} |\dfrac{u_h}{r}(\byi) - \dfrac{u_h}{r}(\bai)| \int_{D_\bai} r^3|\eta_\bai|  d\byi  \\
&\leq C\mathrm{max}_{\byi\in D_\bai} |(\dfrac{u_h}{r}(\byi) - \dfrac{u_h}{r}(\bai))| &&\tn{ by (\ref{eta3}) } \\
&\leq C\mathrm{max}_{\byi\in D_\bai}|\bai-\byi|\|\gradrz\dfrac{u_h}{r}\|_{L^\infty(D_\bai)} \\
&\leq Ch\delta\|\gradrz\dfrac{u_h}{r}\|_{L^\infty(D_\bai)} &&\tn{ since } \byi\in D_\bai.
\end{aligned}
\end{equation}
Recall that, by definition of $A_h$, $u_h=0$ on $\Gamma_0$ and $\dfrac{u_h}{r}\in P_1$.
Therefore, 
\begin{equation}
\begin{aligned}
\Lrnk{R_h^gu_h-u_h}^2 &= \Lrnk{\sum_{i=1}^3\Bigg[\dfrac{S^gu_h}{r}(\ba_i)-\dfrac{u_h}{r}(\bai)\Bigg] r\lambda_i}^2 \\
&\leq    \sum_{i=1}^3|\dfrac{S^gu_h}{r}(\ba_i)-\dfrac{u_h}{r}(\bai)|^2\int_K (r\lambda_i)^2  rdrdz \\
&\leq \sum_{i=1}^3 Ch^2\delta^2\|\gradrz\dfrac{u_h}{r}\|_{L^\infty(D_\bai)}^2h_K^2r_K^3 &&\tn{ by (\ref{dofdof1}) and since $|\lambda_i|\leq 1$ } \\
&\leq Ch^2\delta^2\|\gradrz\dfrac{u_h}{r}\|_{L^\infty(D_K)}^2h_K^2r_K^3 \\
&\leq Ch^2\delta^2\|\gradrz\dfrac{u_h}{r}\|_{L^2_r(D_K)}^2r_K^2 &&\tn{ by (\ref{inv3})  } \\
&\leq C\delta^2\|\dfrac{u_h}{r}\|_{L^2_r(D_K)}^2r_K^2 &&\tn{ by (\ref{inv2})  } \\
&\leq C\delta^2\|u_h\|_{L^2_r(D_K)}^2 &&\tn{ by (\ref{inv4}),  } 
\end{aligned}
\end{equation}
and this completes the proof of (\ref{eq:g2}).

Next, let us prove (\ref{eq:d1}). To do so, it suffices to show that 
\[
\Lrnk{R_h^d\bu}^2 \leq \dfrac{C}{\delta^3}\|\bu\|^2_{C_K}.
\]
By (\ref{Rhd}) and (\ref{pi_div}), we have
\begin{equation} \label{rhduk}
\begin{aligned}
R_h^d\bu |_K 
&= \left[\dfrac{1}{|K|}\int_K \dfrac{(S^d\bu)_\theta - (S^d\bu)_r}{r}\right]
     \begin{pmatrix}
        0 \\
        r\chi_K \\
        0          
     \end{pmatrix}  +
     \sum_{i=1}^3
   \left[ \int_{e_i} 
     \begin{pmatrix}
       (S^d\bu)_r \\
       (S^d\bu)_z     
       \end{pmatrix} \cdot \bn_i ds \right]
         \begin{pmatrix}
        \xi_i^r \\
        \xi_i^r \\
        \xi_i^z    
       \end{pmatrix}.
\end{aligned}
\end{equation}
With this in mind, we show that the following bounds are true.
\begin{align}
\label{dof1}
\int_{e_i} 
     \begin{pmatrix}
       (S^d\bu)_r \\
       (S^d\bu)_z     
       \end{pmatrix} \cdot \bn_i  ds &\leq \left(\dfrac{C}{\delta\sqrt{r_{\ba_j}}}+\dfrac{C}{\delta\sqrt{r_{\ba_k}}}\right)\|\bu\|_{L^2_r(C_K)}, \\
\label{dof2}
\dfrac{1}{|K|}\int_K \dfrac{(S^d\bu)_\theta-(S^d\bu)_r}{r} d\bx &\leq C\Big(\dfrac{1}{h\sqrt{r_{\ba_1}^3}} + \dfrac{1}{h\sqrt{r_{\ba_2}^3}} +\dfrac{1}{h\sqrt{r_{\ba_3}^3}}  \Big) \| \bu \|_{L^2_r(C_K)},
\end{align}
where $C_K$ denotes the convex hull of $D_{\ba_1}, D_{\ba_2}$, and $D_{\ba_3}$. 
We will first derive (\ref{dof1}). Using the $(i,j,k)$-circular permutation notation and the notations $\by_j=(y_j^r, y_j^z)$ and $\by_k=(y_k^r, y_k^z)$, we have
\begin{equation} \label{dof1_1}
\begin{aligned}
|\int_{\be_i} (S^d\bu)_{rz}\cdot\bn_i ds|
&= |\int_{D_\baj}\int_{D_\bak} \kappa_{jk} \int_{[\by_j,\by_k]} 
\bu_{rz}\cdot \bn_i
dSd\by_kd\by_j| &&\textnormal{ by (\ref{def:Sd})} \\
&= |\int_{D_\baj}\int_{D_\bak} \kappa_{jk} \int_0^1
\bu_{rz}((1-s)\by_j+s\by_k)\cdot  
\begin{pmatrix}
     y_j^z-y_k^z \\
     y_k^r-y_j^r    
       \end{pmatrix}
dsd\by_kd\by_j| \\
&\leq Ch \int_{D_\baj}\int_{D_\bak} |\kappa_{jk}|\int_0^1 |\bu_{rz}((1-s)\by_j+s\by_k)| dsd\by_kd\by_j \\
&\leq Ch \int_{D_\baj}\int_{D_\bak} |\kappa_{jk}|(\int_0^{\frac{1}{2}} + \int_{\frac{1}{2}}^1) |\bu_{rz}((1-s)\by_j+s\by_k)| dsd\by_kd\by_j \\
&= A + B,
\end{aligned}
\end{equation}
where we have broken the integral with respect to $s$ into two integrals, one over $s\in[0,1/2]$ ($A$) and the other over $s\in [1/2,1]$ ($B$).
Before we proceed, we introduce another notation. 
For any vertex $\ba$, $r_{\ba,max}$ denotes $\max_{\by\in D_{\ba}}r(\by)$.
It is straightforward to check that 
\begin{equation} \label{threethreethree}
\dfrac{r_{\ba,max}}{r_\ba}\leq 3.
\end{equation}
Now let us consider the case when $0\leq s\leq \frac{1}{2}$  to get an estimate for $A$. 
\begin{equation} \label{dofdof1_2}
\begin{aligned}
A &= Ch \int_{D_\baj}\int_{D_\bak} |\kappa_{jk}|\int_0^{\frac{1}{2}} |\bu_{rz}((1-s)\by_j+s\by_k)| dsd\by_kd\by_j \\ 
&= Ch\int_{D_\bak} |\kappa_k| \int_0^{\frac{1}{2}} \int_{D_\baj} |\kappa_j|  |\bu_{rz}((1-s)\by_j+s\by_k)|  d\by_j dsd\by_k \\
&= Ch\int_{D_\bak} |\kappa_k| \int_0^{\frac{1}{2}} \int_{D_\baj} |r(\by_j)\eta_{\baj}(\by_j)| |r(\by_j) \bu_{rz}((1-s)\by_j+s\by_k)| r(\by_j) d\by_j dsd\by_k \\
&\leq Ch\int_{D_\bak} |\kappa_k| \int_0^{\frac{1}{2}}  \|r\eta_\baj\|_{L_r^2(D_\baj)}  \|r(\by_j)\bu_{rz}((1-s)\by_j+s\by_k)\|_{L_r^2(D_\baj)}dsd\by_k \\
&\leq  Ch\|r\eta_\baj\|_{L_r^2(D_\baj)}r_{\baj,max} \int_{D_\bak} |\kappa_k| \int_0^{\frac{1}{2}} \Big( \int_{D_\baj}(\bu_{rz}((1-s)\by_j+s\by_k))^2 r(\by_j) d\by_j \Big)^{\frac{1}{2}} dsd\by_k \\
&\leq   Ch\|r\eta_\baj\|_{L_r^2(D_\baj)}r_{\baj,max}\int_{D_\bak} |\kappa_k| \int_0^{\frac{1}{2}} \Big(\int_{\bZ_{kj}} \bu_{rz}(\bz)^2 r(\bz) d\bz\Big)^{\frac{1}{2}}dsd\by_k  \tn{ by change of variables. } \\
\end{aligned}
\end{equation}
In the last inequality above, we are using a change of variables from $\by_j$ to $\bz=(1-s)\by_j+s\by_k$. Since $0\leq s\leq 1/2$ we have
\[
r(\by_j) \leq 2(1-s)r(\by_j) \leq 2(1-s)r(\by_j)+2sr(\by_k) \leq 2r((1-s)\by_j+s\by_k) = 2r(\bz)
\]
and  
\[
(1-s)^{-2} \leq 4
\]
to bound the Jacobian. 
Continuing from (\ref{dofdof1_2}), we get
\begin{equation} \label{dof1_2}
\begin{aligned}
A &\leq \dfrac{Chr_{\baj,max}}{h\delta\sqrt{r_\baj^3}} \|\bu_{rz}\|_{L^2_r(C_K)} \int_{D_\bak} |\kappa_k| \int_0^{\frac{1}{2}} dsd\by_k  &&\textnormal{ by (\ref{eta2}) and since } \bZ_{kj}\subset C_K \\
&\leq \dfrac{Chr_{\baj,max}}{h\delta\sqrt{r_\baj^3}} \|\bu_{rz}\|_{L^2_r(C_K)} &&\tn{ by (\ref{eta3}) } \\
&\leq \dfrac{C}{\delta\sqrt{r_\baj}}\|\bu\|_{L^2_r(C_K)} &&\textnormal{ by (\ref{threethreethree}).}
\end{aligned}
\end{equation}
We get a similar results as (\ref{dof1_2}) for $B$ as well, and this proves (\ref{dof1}).

To prove (\ref{dof2}) let us first introduce the following notation:
\[
T_l=\{\bx\in K: \lambda_l(x)>\frac{1}{3} \} \tn{ for } 1\leq l\leq 3. 
\]
Since
\[
\dfrac{1}{|K|}\int_K \dfrac{(S^d\bu)_\theta-(S^d\bu)_r}{r} d\bx \leq \dfrac{1}{|K|}\sum_{l=1}^3|\int_{T_l} \dfrac{(S^d\bu)_\theta-(S^d\bu)_r}{r} d\bx|, 
\]
it suffices to show that the bound holds for each $\dfrac{1}{|K|}|\int_{T_l} \dfrac{(S^d\bu)_\theta-(S^d\bu)_r}{r} d\bx|$. 
\begin{equation*}
\begin{aligned}
&|\int_{T_1}\dfrac{(S^d\bu)_\theta-(S^d\bu)_r}{r}d\bx_1| \\
&= |\int_{T_1}\triint \kappa_{123}\dJ \Big[\dfrac{u_\theta-u_r}{r} \Big](\xt) d\by_3d\by_2d\by_1d\bx| \quad\quad\textnormal{ by (\ref{def:Sd})} \\
%&\leq \iDaj\iDak |\kappa_{23}| \Big(\int_{T_1}\iDai r(\by_1)|\eta_{\ba_1}(\by_1)|  r(\by_1)\left|\dfrac{u_\theta-u_r}{r} (\xt)\right| r(\by_1) \left|\dJ\right| d\by_1d\bx \Big) d\by_3d\by_2, \\
&\leq \iDaj\iDak |\kappa_{23}| \Big(\int_{T_1}\iDai r(\by_1)|\eta_{\ba_1}(\by_1)|  r(\by_1)\left|\dfrac{u_\theta-u_r}{r} (\xt)\right| r(\xt) \left|\dJ\right| d\by_1d\bx \Big) d\by_3d\by_2 \\
&\leq \iDaj\iDak |\kappa_{23}| \Big(\int_{T_1} \| r\eta_{\ba_1} \|_{L_r^2(D_{\ba_1})} \| (u_\theta-u_r)(\xt) \|_{L_r^2(D_{\ba_1})} \left|\dJ\right| d\bx \Big) d\by_3d\by_2 \\ %\tn{since $r(\by_1)<3r(\xt)$ on $T_1$, } \\
&\leq \| r\eta_{\ba_1} \|_{L_r^2(D_{\ba_1})} \sqrt{|T_1||\dJ|}  \\ 
&\quad\quad\quad \cdot \iDaj\iDak |\kappa_{23}| \Big( \int_{T_1}\iDai ((u_\theta-u_r)(\xt))^2 r(\by_1) | \dJ | d\by_1d\bx  \Big)^{1/2}d\by_3d\by_2 \\
&\leq  \| r\eta_{\ba_1} \|_{L_r^2(D_{\ba_1})} \sqrt{|T_1||\dJ|}  \\
&\quad\quad\quad \cdot\iDaj\iDak |\kappa_{23}| \Big(\iDai\int_{\widetilde{T}_1} ((u_\theta-u_r)(\bz))^2 r(\bz)  d\bz d\by_1  \Big)^{1/2}d\by_3d\by_2  \tn{ by change of variables } \\
&\leq \dfrac{C}{\sqrt{\rho^2 r_{\ba_1}^3}} \sqrt{Ch^2} \sqrt{\rho^2}\| u_\theta-u_r \|_{L^2_r(C_K)} \quad\quad\tn{ since $(area(D_{\ba_1}))^{1/2}=\sqrt{\pi\rho^2}$ } \\
&\leq \dfrac{Ch}{\sqrt{r_{\ba_1}^3}}\| \bu \|_{L^2_r(C_K)}.
\end{aligned}
\end{equation*}
The change of variables used above in particular is from $\bx\in T_1$ to $\bz=\xt$.
The second inequality above is true, since 
\[
r(\by_1)=\dfrac{1}{\lambda_1(\bx)}\cdot\lambda_1(\bx)\cdot r(\by_1) < 3r(\lambda_1(\bx)\by_1) \leq 3r(\xt),
\]
as $1/\lambda_1(\bx)<3$ for all $\bx\in T_1$. 
We are also using above that $\widetilde{T}_1$, the image of $T_1$ under the map $\bx\rightarrow \xt$, satisfies $\tilde{T}_1\subset C_K$ and that
\[
|\dJ|=\dfrac{|\tilde{T}_1|}{|T_1|}\leq C.
\]

Similar results can be shows for $T_2$ and $T_3$, and we get
\[
\dfrac{1}{|K|}\int_K \dfrac{(S^d\bu)_\theta-(S^d\bu)_r}{r} d\bx \leq C\Big(\dfrac{1}{h\sqrt{r_{\ba_1}^3}} + \dfrac{1}{h\sqrt{r_{\ba_2}^3}} +\dfrac{1}{h\sqrt{r_{\ba_3}^3}}  \Big) \| \bu \|_{L^2_r(C_K)}.
\]
Let us now complete the proof of (\ref{eq:d1}) by using (\ref{dof1}) and (\ref{dof2}). 
It is true that
\begin{equation} \label{aa}
\begin{aligned}
\int_K \big(\sum_{i=1}^3 \int_{\be_i} ((S^d\bu)_{rz}\cdot \bn_i  dS) \begin{bmatrix}
\xi_i^r \\
\xi_i^r \\
\xi_i^z
\end{bmatrix}\big)^2 rdrdz 
&\leq  \sum_{i=1}^3 |\int_{\be_i} (S^d\bu)_{rz}\cdot \bn_i  dS|^2 
\int_K \begin{bmatrix}
\xi_i^r \\
\xi_i^r \\
\xi_i^z
\end{bmatrix}^2 rdrdz \\
&\leq \dfrac{Cr_K}{\delta^2r_\ba}\Big \|\bu\|_{L^2_r(C_K)}^2 &&\tn{ by (\ref{dof1}) } \\
&\leq \dfrac{C}{\delta^3}\|\bu\|_{L^2_r(C_K)} &&\textnormal{ by (\ref{rkra}).}
\end{aligned}
\end{equation}
In the second to the last inequality above, we are using $r_\ba$ to denote the minimal value of $r_\bai, r_\baj$, and $r_\bak$,
and we are using the fact that $(\xi_i^r,\xi_i^z)$ is the local basis functions for the lowest order Raviart Thomas space in two-dimensions, and thus satisfies
\begin{equation} \label{RT}
\| (\xi_i^r,\xi_i^z) \|^2_{L^2_r(K)} \leq Cr_K.
\end{equation}
It is also true that 
\begin{equation} \label{aaa}
\begin{aligned}
&\int_K \big( \dfrac{1}{|K|}\int_K \dfrac{(S^d\bu)_\theta-(S^d\bu)_r}{r} d\bx r\chi_K \big)^2 rdrdz \\
&\leq r_K^3  \big( \dfrac{1}{|K|}\int_K \dfrac{(S^d\bu)_\theta-(S^d\bu)_r}{r} d\bx \big)^2 \int_K \chi_K^2 drdz \\
&\leq \dfrac{Cr_K^3}{h^2r_\ba^3}\|\bu\|_{L^2_r(C_K)}\cdot Ch^2 &&\tn{ by (\ref{dof2})}  \\
&\leq \dfrac{C}{\delta^3}\|\bu\|_{L^2_r(C_K)} &&\tn{ by (\ref{rkra}). } 
\end{aligned}
\end{equation}
The proof of (\ref{eq:d1}) is completed by (\ref{aa}), (\ref{aaa}), and (\ref{rhduk}).

Finally, we prove (\ref{eq:d2}). 
Let $\bu_h\in \bC_h$ be written as $\bu_h=(u_h^r,u_h^\theta,u_h^z)^T$. We will find bounds for the degrees of freedom that defines $R_h^d\bu_h-\bu_h$.
Let $L_{jk}=[\baj,\bak,\by_k,\by_j,\baj]$ (See Figure~\ref{fig:Larea}.) It should be clear that
\[
\mathop\mathrm{area}(L_{jk}) \leq Ch(h\delta)
\]
and
\[
\mathop\mathrm{length}([\baj,\by_j]) \leq Ch\delta \quad\quad \mathop\mathrm{length}([\bak,\by_k]) \leq Ch\delta.
\]
Also, we use $\dive\nolimits_{rz}$ to denote the usual divergence operator in the $(r,z)$-plane, i.e.,
\[
\dive\nolimits_{rz}(v_r,v_z) = \partial_rv_r + \partial_zv_z.
\]
\begin{figure}
  \centering
    \begin{tikzpicture} 
    \coordinate [label=right: $\ba_k$] (ak) at (0,0);
    \coordinate [label=left:$\ba_j$] (aj) at (-5,0);
    \coordinate [label=below:$\by_k$] (yk) at (-0.4,-0.35);
    \coordinate [label=left:$\by_j$] (yj) at (-5.2,0.5);

    \draw [fill=yellow!30] (aj)--(ak)--(yk)--(yj)--cycle 
          node[anchor=west] {\qquad\qquad$L_{jk}$};

    \draw (ak) circle (1cm);
    \draw (aj) circle (1cm);

    \fill (ak) circle (2pt);
    \fill (aj) circle (2pt);

    \fill (yk) circle (2pt);    
    \fill (yj) circle (2pt);

    \end{tikzpicture}
  \caption{$L_{jk}$}
  \label{fig:Larea}
\end{figure}

It follows that
\begin{equation}  \label{aa2}
\begin{aligned}
&\Big| \iDajj\iDakk \kappa_{jk} \Big( \int_{[\by_j,\by_k]} (\bu_h)_{rz}\cdot\bn_i dS - \int_{[\baj,\bak]} (\bu_h)_{rz}\cdot\bn_i dS \Big) d\by_kd\by_j\Big| \\
&\leq \iDajj\iDakk |\kappa_{jk}| \Big(|\int_{L_{jk}}\dive(\bu_h)_{rz} dA| + |\int_{[\baj,\by_j]}(\bu_h)_{rz}\cdot\bn_i dS| + |\int_{[\bak,\by_k]}(\bu_h)_{rz}\cdot\bn_i dS|  \Big)d\by_kd\by_j \\ &\quad\quad\tn{ by the Divergence Theorem } \\
&\leq \iDajj\iDakk |\kappa_{jk}| \Big( Ch(h\delta)\| \dive(\bu_h)_{rz} \|_{L^\infty(C_K)} + Ch\delta\| (\bu_h)_{rz} \|_{L^\infty(C_K)} \Big)d\by_kd\by_j \\
&\leq Ch\delta\| (\bu_h)_{rz} \|_{L^\infty(C_K)} \quad\quad\tn{ by (\ref{eta3}) and an inverse inequality.}  
\end{aligned}
\end{equation}
Therefore, 
\begin{equation} \label{aa3}
\begin{aligned}
\int_K \Big(\sum_{i=1}^3 (\int_{\be_i} (R_h^d\bu_h-\bu_h)_{rz}\cdot\bn_i) \begin{bmatrix}
\xi_i^r \\
\xi_i^r \\
\xi_i^z
\end{bmatrix}\Big)^2 rdrdz &\leq Ch^2\delta^2\| (\bu_h)_{rz} \|_{L^\infty(C_K)}r_K &&\tn{ by (\ref{aa2}) and (\ref{RT}) } \\
&\leq C\delta^2\|\bu_h\|_{L^2_r(C_K)} &&\tn{ by (\ref{inv3}). }
\end{aligned}
\end{equation}
Next, 
\begin{equation} \label{rhdrhd} 
\begin{aligned}
&\dfrac{1}{|K|}\int_K \dfrac{(R_h^d\bu_h - \bu_h)_\theta-(R_h^d\bu_h - \bu_h)^r}{r} d\bx \\
&=\dfrac{1}{|K|} \int_K\triint |\kappa_{123}| \Big( \dJ (\dfrac{u_h^\theta-u_h^r}{r})(\xt) - (\dfrac{u_h^\theta-u_h^r}{r})(\bx) \Big)d\by_3d\by_2d\by_1d\bx \\
&=\dfrac{1}{|K|} \triint |\kappa_{123}| \Big( \int_K \dJ (\dfrac{u_h^\theta-u_h^r}{r})(\xt) d\bx - \int_K(\dfrac{u_h^\theta-u_h^r}{r})(\bx)d\bx \Big)d\by_3d\by_2d\by_1 \\
&= \dfrac{1}{|K|}\triint |\kappa_{123}| \Big( \int_{\tilde{K}_y}  (\dfrac{u_h^\theta-u_h^r}{r})(\bz)d\bz - \int_K (\dfrac{u_h^\theta-u_h^r}{r})(\bx) d\bx \Big)d\by_3d\by_2d\by_1 \\ &\quad\quad\tn{ by change of variables from $\bx$ to $\bz=\xt$ } \\
&\leq  \dfrac{1}{|K|}\| \dfrac{u_h^\theta-u_h^r}{r}\|_{L^\infty(C_K)}h(h\delta) \quad\quad\tn{ by (\ref{eta3}) } \\
&\leq \delta \|\dfrac{u_h^\theta-u_h^r}{r}\|_{L^\infty(C_K)}.
\end{aligned}
\end{equation}
In the second to the last inequality above, we are also using the fact that 
\[
\mathop\mathrm{area}((K\backslash\tilde{K}_y) \cup (\tilde{K}_y\backslash K))\leq Ch(h\delta), 
\]
which is clear from Figure~\ref{fig:Kt}.  

Finally, 
\begin{equation} \label{aaa3}
\begin{aligned}
&\int_K \big(\dfrac{1}{|K|}\int_K \dfrac{(R_h^d\bu_h - \bu_h)_\theta-(R_h^d\bu_h - \bu_h)^r}{r} d\bx r\chi_K \big)^2 rdrdz \\
 &\leq h^2r_K^3\delta^2 \left\|\dfrac{u_h^\theta-u_h^r}{r}\right\|_{L^\infty(C_K)}^2 &&\tn{ by (\ref{rhdrhd}) } \\ 
 &\leq \delta^2r_K^2\left\|\dfrac{u_h^\theta-u_h^r}{r}\right\|_{L^2_r(C_K)}^2 &&\tn{ by (\ref{inv3})} \\  
 &\leq \delta^2 \| u_h^\theta-u_h^r \|_{L^2_r(C_K)}^2 &&\tn{ by (\ref{inv4})} \\  
 &\leq C\delta^2\| \bu_h \|_{L^2_r(C_K)}^2. 
\end{aligned}
\end{equation}
Hence, the proof of (\ref{eq:d2}) is completed by (\ref{aa3}), (\ref{aaa3}), and (\ref{pi_div}).

\begin{figure}
  \centering
  \begin{tikzpicture}
    \coordinate [label=right: $\ba_k$] (ak) at (0,1);
    \coordinate [label=left:$\ba_j$] (aj) at (-5,0);
    \coordinate [label=left:$\ba_i$] (ai) at (-2,-3);

    \coordinate [label=above:$\by_k$] (yk) at ($(ak)+(-0.4,0.2)$);
    \coordinate [label=left:$\by_j$]  (yj) at (-5.2,0.5);
    \coordinate [label=below:$\by_i$] (yi) at ($(ai)+(0.6,0.3)$);
    
    \coordinate (yjk) at ($(yk)!.5!(yj)$);
    \coordinate (aij) at ($(ai)!.5!(aj)$);

    \node[anchor=south] at (yjk) {\textcolor{black}{$\Kt_{\by}$}};
    \node[anchor=east] at (aij) {$K$};

    \draw (ak) circle (1cm);
    \draw (aj) circle (1cm);
    \draw (ai) circle (1cm);

    \fill (ak) circle (2pt);
    \fill (aj) circle (2pt);
    \fill (ai) circle (2pt);

    \fill (yk) circle (2pt);    
    \fill (yj) circle (2pt);    
    \fill (yi) circle (2pt);    
   
    % K and Kt
    \draw [fill=green, fill opacity=0.5] (aj)--(ak)--(ai)--cycle ;
    \draw [purple, fill=yellow!50, fill opacity=0.5] (yj)--(yk)--(yi)--cycle;

  \end{tikzpicture}
  \caption{$K$ is mapped to $\Kt_{\by}$ under the map $\bx \mapsto
    \xt$.}
  \label{fig:Kt}
\end{figure}
\end{proof}

\begin{remark} \label{good}
It is worth mentioning how we are using $\kappa_i$ in the proof of Lemma~\ref{lem:Rbounds}. The three $r$'s in the definition of $\kappa_i=r^3\eta_\bai$ are each being used for a different purpose.
One $r$ is being multiplied to $\eta_\bai$ so that we can use (\ref{eta2}), another $r$ is being multiplied to the integrand that appears next to $\kappa_i$,
and the last $r$ is being left in the integral so that the integral is still an inner product in the weighted $L^2$-space, so that the resulting norms in the next step after applying the Cauchy-Schwarz inequality
continue to be $L^2_r$-norms. (See (\ref{two}) for example.) 
\end{remark} 

The following Lemma is a straightforward result followed by (\ref{eq:g2}), (\ref{eq:c2}), (\ref{eq:d2}), (\ref{eq:o2}), and a standard Neumann series argument.
\begin{lemma} \label{lem:inverse}
There are operators 
$$
J^g_h: A_h \rightarrow A_h, \quad J^c_h: \bB_h \rightarrow \bB_h, \quad J^d_h: \bC_h \rightarrow \bC_h, \quad J^o_h: D_h \rightarrow D_h,
$$
and $\delta_1>0$ such that for all $0<\delta<\delta_1$, the operators 
$R^g_h|_{A_h}, R^c_h|_{\bB_h}, R^d_h|_{\bC_h}$, and $R^o_h|_{D_h}$ are
invertible, their inverses are
$J^g_h, J^c_h, J^d_h$, and $J^o_h$, resp., and their operator norms satisfy
\[
\| J^g_h \|_{L_r^2(\Omega)}\le 2,\quad
\| J^c_h \|_{L_r^2(\Omega)} \le 2,\quad
\| J^d_h \|_{L_r^2(\Omega)} \le 2,\quad
\| J^o_h \|_{L_r^2(\Omega)} \le 2.
\]
Furthermore, these inverse operators satisfy the commuting diagram property as well. 
\end{lemma}

Finally, we are ready to define the commuting smoothed projectors that are uniformly bounded in the $L^2_r$-norm.
For the rest of the paper, we assume that $\delta\in (0,\delta_1]$, where $\delta_1$ is given as in Lemma~\ref{lem:inverse}.
Recall that $\bL_r^2(\Omega)$ denotes $L^2_r(\Omega)\times L^2_r(\Omega)\times L^2_r(\Omega)$.
\begin{definition}
  \label{def:proj}
  Define $\Pi^g_h : L^2_r(\Omega) \rightarrow A_h$, $\Pi^c_h : \bL^2_r(\Omega)
  \rightarrow \bB_h$, $\Pi^d_h : \bL^2_r(\Omega) \rightarrow \bC_h$ and $\Pi^o_h : L^2_r(\Omega) \rightarrow D_h$ by
  \[
  \Pi^g_h = J^g_hR^g_h,
  \qquad
  \Pi^c_h = J^c_hR^c_h,
  \qquad
   \Pi^d_h = J^d_hR^d_h,
  \qquad
  \Pi^o_h = J^o_hR^o_h.
  \]
\end{definition}

The following theorem summarizes the properties of these projectors.

\begin{theorem}
\label{thm:main}
The above operators are projectors and have the following properties:
\begin{enumerate}
\item \label{item:cnty}
  {\em Continuity.} There exists a $C>0$ such that 
  \begin{align*}
    \| \Pi_h^g u\|_{L_r^2(\Omega)} & \le C \| u \|_{L_r^2(\Omega)},
    && \forall u \in L_r^2(\Omega),
    \\
    \| \Pi_h^c \bu\|_{L_r^2(\Omega)} & \le C \| \bu \|_{L_r^2(\Omega)},
    && \forall \bu \in \bL^2_r(\Omega),
    \\
      \| \Pi_h^d \bu\|_{L_r^2(\Omega)} & \le C \| \bu \|_{L_r^2(\Omega)},
    && \forall \bu \in \bL^2_r(\Omega),
    \\
    \| \Pi_h^o u\|_{L_r^2(\Omega)} & \le C \| u \|_{L_r^2(\Omega)},
    && \forall u \in L_r^2(\Omega).
  \end{align*}

\item {\em Commutativity.} The operators satisfy the following
  commuting diagram:
\begin{equation}
\begin{CD}
L^2_r(\Omega)   @>\mathbf{grad}^n_{rz}>> \bL^2_r(\Omega)   @>\bcurl\nolimits^n\nolimits_{rz}>> \bL^2_r(\Omega)   @>\dive^n_{rz}>>   L^{2}_{r}(\Omega) \\ %\cap L^2_{1/r}(\Omega)  \\
@VV \Pi_h^g V                                   @VV \Pi_h^c V                           @VV \Pi_h^d V @ VV \Pi_h^o V  \\
A_h @> \grad_{rz}^n >> \bB_h @> \bcurl\nolimits^n\nolimits_{rz} >> \bC_h @> \dive_{rz}^n >> D_h
\end{CD}
\end{equation}

\item \label{item:apprx} 
  {\em Approximation.} 
\begin{align*}
  \left\| u - \Pi^g_hu \right\|_{L^2_r(\Omega)}  &\leq C\inf_{u_h\in A_h}\left\| u - u_h \right\|_{L^2_r(\Omega)} , \\
  \left\| \bu - \Pi^c_h\bu \right\|_{L^2_r(\Omega)}  &\leq C\inf_{\bu_h\in \bB_h}\left\| \bu - \bu_h \right\|_{L^2_r(\Omega)} , \\ 
   \left\| \bu - \Pi^d_h\bu \right\|_{L^2_r(\Omega)}  &\leq C\inf_{\bu_h\in \bC_h}\left\| \bu - \bu_h \right\|_{L^2_r(\Omega)} , \\
   \left\| u - \Pi^o_h u \right\|_{L^2_r(\Omega)}  &\leq C\inf_{u_h\in D_h}\left\| u - u_h \right\|_{L^2_r(\Omega)} . 
\end{align*}
\end{enumerate}
\end{theorem}

\begin{proof}
We will prove all items for $\Pi_h^g$ only as the remaining statements can be proved in a similar way.
First of all, $\Pi_h^g$ is indeed a projector, since
\[
(\Pi_h^g)^2v = (J^g_hR^g_h)^2v = J^g_hR^g_hJ^g_hR^g_hv = J^g_h(R^g_h|_{A_h}J^g_h)R^g_hv =  J^g_hR^g_hv = \Pi_h^gv.
\]
Furthermore, by (\ref{eq:g1}) and Lemma~\ref{lem:inverse}, we also have
  \[
  \| \Pi_h^g \|_{L_r^2(\Omega)}  \le C \| J_h^g \|_{L_r^2(\Omega)} \| R_h^g \|_{L_r^2(\Omega)} \le 2C/\delta^5,
    \]
 which proves the continuity estimate. 
 Commutativity follows by Proposition~\ref{prop_comm} and the fact that the inverse operators in Lemma~\ref{lem:inverse} also satisfy a commuting diagram property. 
 \begin{comment}
 For example,
 if $J_h^gv_h=w_h$, then
 \[
 \grad\nolimits_{rz}\nolimits^n(J_h^gv_h) =  \grad\nolimits_{rz}\nolimits^nw_h = J_h^g(R_h^g \grad\nolimits_{rz}\nolimits^nw_h) =J_h^g( \grad\nolimits_{rz}\nolimits^n R_h^gw_h) = J_h^g( \grad\nolimits_{rz}\nolimits^nv_h).
 \]
 \end{comment}
 Lastly, since $ \Pi^g_h u_h = u_h$, we have
  \begin{align*}
    \left\| u - \Pi^g_hu \right\|_{L^2_r(\Omega)} = \| u - u_h - \Pi^g_h(u-u_h) \|_{L^2_r(\Omega)} \leq (1+C)\| u - u_h  \|_{L^2_r(\Omega)},
  \end{align*}
  for any $u_h\in A_h$, and therefore
  \[
   \left\| u - \Pi^g_hu \right\|_{L^2_r(\Omega)}  \leq C\inf_{u_h\in A_h}\left\| u - u_h \right\|_{L^2_r(\Omega)}.
  \]
\end{proof}

It is clear from Theorem~\ref{thm:main} that the projections constructed in the previous section are certainly uniformly $V$-bounded cochain projections by (\ref{bounded}),
and this verifies Theorem~\ref{main2}, the main result of this paper. 

\textcolor{black}{\begin{remark}
We note here that the constant $C$ appearing in Theorem~\ref{thm:main}~items (\ref{item:cnty}) and (\ref{item:apprx}) is independent of the Fourier mode $n$.
Item~(\ref{item:apprx}), however, provides abstract error estimates, so we investigate how the concrete error estimates are affected by $n$ through a numerical example in the next section.  
\end{remark}}

\section{Numerical Results} \label{numerical}

In this section, we present numerical results for the weighted mixed formulation of the Hodge Laplacian problem that we investigated in this paper.
We present numerical results of the weighted mixed formulation of the Hodge Laplacian for $k=0, 1, 2$, and $3$. For all examples presented below,
we choose the domain to be the unit square in $\RRR^2_+$ with vertices $(0,0), (1,0), (1,1),$ and $(0,1)$. 

\vspace{\baselineskip}

\underline{Weighted Hodge Laplacian Problem with $k=3$}

\vspace{\baselineskip}

The problem of interest in this case can be stated as follows:

find $(\boldsymbol{\sigma}_h,u_h)\in  \bC_h \times D_h$ such that
\begin{equation} \label{example1}
\begin{aligned}
(\boldsymbol{\sigma}_h,\boldsymbol{\tau}_h)_{L^2_r(\Omega)} - (\dive\nolimits_{rz}^1\boldsymbol{\tau}_h, u_h)_{L^2_r(\Omega)}  &= 0, &&\textnormal{ for all } \boldsymbol{\tau}_h\in \bC_h, \\
(\dive\nolimits_{rz}^1\boldsymbol{\sigma}_h,v_h)_{L^2_r(\Omega)}  &= (f,v_h)_{L^2_r(\Omega)}  &&\textnormal{ for all } v_h\in D_h.
\end{aligned}
\end{equation}

For computer implementation of the mixed method, we need to assemble
the matrix representations of the operators $K_h: \bC_h \mapsto \bC_h'$
and $L_h: \bC_h \mapsto D_h'$ defined by
\begin{align*}
K_h \bv_h (\bw_h) &= (\bv_h, \bw_h)_{L^2_r(\Omega)}  \textnormal{ for all } \bv_h, \bw_h \in \bC_h \\ 
L_h \bv_h (s_h) &= -(\dive\nolimits_{rz}^1 \bv_h, s_h)_{L^2_r(\Omega)}  \textnormal{ for all } \bv_h \in \bC_h, s_h \in D_h.
\end{align*}
Let $\mathtt{K}$ and $\mathtt{L}$ denote the matrix
representations of $K_h$ and $L_h$ respectively, in terms of the
standard local bases for  $\bC_h$ and $D_h$.
Then~(\ref{example1}) can be rewritten as the linear system
\begin{align*}
\mathtt{K}\underline{\sigma_h} + \mathtt{L^T}\underline{u_h}
&= 0, \\
-\mathtt{L}\;\underline{\sigma_h} 
& = \underline{f_h},
\end{align*}
where $\underline{\sigma_h}$ and $\underline{u_h}$ denote the vectors of
coefficients in the basis expansions of $\boldsymbol{\sigma}_h$ and $u_h$,
respectively.  The vector $\underline{f_h}$ is computed from the right
hand side of (\ref{example1}) as usual. In practice, we compute
$\underline{u_h}$ and $ \underline{\sigma_h}$ by solving
\begin{align*}
\mathtt{M} \underline{u_h}&= \underline{f_h}, \\
\mathtt{K} \underline{\sigma_h} &= \underline{g_h},
\end{align*}     
where $\mathtt{M} = \mathtt{L}\mathtt{K^{-1}}\mathtt{L^T}$ and
$\underline{g_h} = -\mathtt{L^T}\underline{u_h}$. Both these systems can be
solved via the conjugate gradient method as $\mathtt{M}$ and
$\mathtt{K}$ are symmetric and positive definite. Note that when
solving the first equation, for each application of $\mathtt{M}$, we
use another inner conjugate gradient iteration to obtain the result of
multiplication by $\mathtt{K^{-1}}$.

In Table~\ref{ex1}, we report the $L^2_r(\Omega)$-norm of the observed errors.
$u_l$ and $\boldsymbol{\sigma}_l$ are used to denote the approximation of $u$ and $\boldsymbol{\sigma}$ respectively at mesh level $l$.
We choose the right hand side data function $f$ so that
the exact solution $(\boldsymbol{\sigma}, u)$ is %\in \bH_r(\dive\nolimits^1, \Omega)\times H_r( \grad\nolimits^1, \Omega)$ is
\begin{align*}
u &= \sin(\pi z)(r^2-r), \\  
\boldsymbol{\sigma}  &= (-\sin(\pi z)(2r-1),  -\sin(\pi z)(r-1), -\pi\cos(\pi z)(r^2-r))^T.
\end{align*}

\begin{table} 
\caption{$k=3$ case with Fourier mode $n=1$}
\label{ex1} 
\begin{center}
 \begin{tabular}{|c| c | c| c| c |} 
 \hline
 mesh level $l$ & $\left\| u-u_l \right\|_{L^2_r(\Omega)}$ & rate & $\left\| \boldsymbol{\sigma}-\boldsymbol{\sigma}_l \right\|_{L^2_r(\Omega)}$  & rate  \\ [0.5ex] 
 \hline
1 & 4.265e-02    &          &   1.922e-01   &         \\
2 & 2.318e-02    &  0.88 &   1.018e-01  &  0.92   \\
3 & 1.188e-02    &  0.96 &    5.198e-02  & 0.97   \\
4 & 5.979e-03    &  0.99 &   2.616e-02   & 0.99  \\
5 & 2.995e-03    &  1.00 &    1.311e-02   &  1.00 \\
6 & 1.498e-03    &  1.00 &      6.556e-03 &  1.00   \\ 
7 & 7.492e-04    &  1.00 &    3.278e-03 & 1.00 \\  
 \hline
\end{tabular} 
\end{center}
\end{table}

\vspace{3\baselineskip}

\color{black}
\underline{Weighted Hodge Laplacian Problem with $k=2$}

\vspace{\baselineskip}

We next consider the case when $k=2$ with Fourier mode $n=3$. We are interested in finding $(\boldsymbol{\sigma}_h, \bu_h)\in \bB_h\times \bC_h$ that satisfies
\begin{align*}
(\boldsymbol{\sigma}_h, \boldsymbol{\tau}_h)_r - (\bcurlrz^3\boldsymbol{\tau}_h, \bu_h)_r &= 0 &&\textnormal{ for all } \boldsymbol{\tau}_h\in \bB_h, \\
(\bcurlrz^3\boldsymbol{\sigma}_h, \bv_h)_r + (\dive\nolimits_{rz}\nolimits^3 \bu_h, \dive\nolimits_{rz}\nolimits^3 \bv_h)_r &= (\bff, \bv_h)_r &&\textnormal{ for all } \bv_h\in \bC_h.
\end{align*}
The implementation is done in the usual way similar as the $k=3$ case, and
we choose $\bff$ so that the exact solution $(\boldsymbol{\sigma}, \bu)$ is 
\begin{align*}
\bu &= (3r(r-1), -3r^2+2r, 0)^T, \\
\boldsymbol{\sigma} &= (0, 0, r^2(r-1))^T.
\end{align*}
In Table~\ref{ex4}, we report the $L^2_r(\Omega)$-errors for this problem.

\begin{table} 
\caption{$k=2$ case with Fourier mode $n=3$}
\label{ex4} 
\begin{center}
 \begin{tabular}{|c| c | c| c| c |} 
 \hline
 mesh level $l$ & $\left\| \bu-\bu_l \right\|_{L^2_r(\Omega)}$ & rate & $\left\| \boldsymbol{\sigma}-\boldsymbol{\sigma}_l \right\|_{L^2_r(\Omega)}$  & rate  \\ [0.5ex] 
 \hline
1  & 3.416e-02  &          & 2.295e-01  & \\
2  & 2.865e-02  & 0.25  & 1.367e-01  & 0.75 \\
3  & 2.033e-02  & 0.50  & 7.075e-02  & 0.95 \\
4  & 1.105e-02  & 0.88  & 3.583e-02  & 0.98 \\
5  & 5.652e-03  & 0.97  & 1.800e-02  & 0.99 \\
6  & 2.845e-03  & 0.99  & 9.015e-03  & 1.00 \\
7  & 1.425e-03  & 1.00  & 4.510e-03  & 1.00 \\
 \hline
\end{tabular} 
\end{center}
\end{table}

\vspace{\baselineskip}

\underline{Weighted Hodge Laplacian Problem with $k=1$} 

\vspace{\baselineskip}

The discrete problem for $k=1$ with Fourier mode $n=2$ can be stated as follows:

Find $(\sigma_h, \bu_h)\in A_h\times \bB_h$ that satisfies
\begin{align*}
(\sigma_h, \tau_h)_r - (\gradrz^2\tau_h, \bu_h)_r &= 0 &&\textnormal{ for all } \tau_h\in A_h, \\
(\gradrz^2\sigma_h, \bv_h)_r + (\curlrz^2 \bu_h, \curlrz^2 \bv_h)_r &= (\bff, \bv_h)_r &&\textnormal{ for all } \bv_h\in \bB_h.
\end{align*}
We choose $\bff$ so that the exact solution $(\bu,\sigma)$ is
\begin{align*}
\bu &= (r^3(r-1), 0, 0)^T, \\
\sigma &= -5r^3+4r^2.
\end{align*}
The error for this case is reported in Table~\ref{ex3}.
\begin{table} 
\caption{$k=1$ case with Fourier mode $n=2$}
\label{ex3} 
\begin{center}
 \begin{tabular}{|c| c | c| c| c |} 
 \hline
 mesh level $l$ & $\left\| \bu-\bu_l \right\|_{L^2_r(\Omega)}$ & rate & $\left\| \sigma-\sigma_l \right\|_{L^2_r(\Omega)}$  & rate  \\ [0.5ex] 
 \hline
1  & 3.368e-02  &          & 1.147e-01  &        \\
2  & 2.484e-02  & 0.44  & 6.520e-02  & 0.81 \\
3  & 1.749e-02  & 0.51  & 1.964e-02  & 1.73  \\
4  & 9.658e-03  & 0.86  & 5.206e-03  & 1.92  \\
5  & 4.963e-03  & 0.96  & 1.326e-03  & 1.97  \\
6  & 2.501e-03  & 0.99  & 3.335e-04  & 1.99  \\
7  & 1.254e-03  & 1.00  & 8.351e-05  & 2.00  \\
 \hline
\end{tabular} 
\end{center}
\end{table}

\vspace{\baselineskip}

\color{black}
\underline{Weighted Hodge Laplacian Problem with $k=0$}

\vspace{\baselineskip}

We fix the Fourier mode to be $n=1$ in this example. 
Let $Q_h: H_r(\grad^1,\Omega) \rightarrow A_h$ be a projection that satisfies
\begin{equation*}
(\gradrz^1 Q_h u, \gradrz^1 v_h)_{L^2_r(\Omega)}  = (\gradrz^1 u, \gradrz^1 v_h)_{L^2_r(\Omega)}  \textnormal{ for all } v_h\in A_h.
\end{equation*}
Note that, the right-hand-side of the above problem is in a different form compared to that appearing on the right-hand-side of the
weighted mixed formulation of the Hodge Laplacian for $k=0$. Also recall that the nullspace of the $\gradrz^1$ operator is trivial. 
In Table~\ref{ex2}, we report the $L^2_r(\Omega)$-norm of the observed error for different $u$-values that are indicated in the chart. 
As expected, one can see the increase of the order of convergence as the $r$-degree grows higher. 

\footnotesize
\begin{table}
\caption{$k=0$ case with Fourier mode $n=1$}
\label{ex2}
\begin{center}
 \begin{tabular}{| c | c | c || c| c || c | c || c | c |} 
 \hline
 \multicolumn{1}{|c|}{} &\multicolumn{2}{c||}{$u=r^{1/2}$} &\multicolumn{2}{c||}{$u=r^{2/3}$} &\multicolumn{2}{c||}{$u=r^{5/6}$} &\multicolumn{2}{c|}{$u=r\sin z$} \\\cline{2-9}
   \hline
   mesh level    & $\left\| u-Q_h u\right\|_{L^2_r(\Omega)}$ & rate & $\left\| u-Q_h u\right\|_{L^2_r(\Omega)}$  & rate & $\left\| u-Q_h u\right\|_{L^2_r(\Omega)}$  & rate & $\left\| u-Q_h u\right\|_{L^2_r(\Omega)}$ & rate\\ 
 \hline
1 & 8.578e-03    &             &   4.086e-03 &           & 1.474e-03 &.         & 1.116e-03 &  \\
2 & 3.306e-03    &  1.38   &    1.437e-03   & 1.51 & 4.789e-04 & 1.62  & 2.904e-04 & 1.94 \\
3 & 1.225e-03    &   1.43 &     4.802e-04  & 1.58 & 1.465e-04 & 1.71 &  7.349e-05 & 1.98  \\
4 & 4.464e-04    &  1.46 &      1.569e-04   &  1.61 & 4.351e-05 & 1.75 & 1.843e-05 &  2.00 \\
5 & 1.615e-04    &  1.47 &       5.074e-05  &  1.63 & 1.272e-05 & 1.77 & 4.609e-06 & 2.00 \\
6 & 5.825e-05  &   1.47  &      1.631e-05 & 1.64 & 3.685e-06 & 1.79 & 1.152e-06  & 2.00 \\
7 & 2.096e-05   &   1.47   &   5.225e-06  & 1.64 & 1.061e-06 & 1.80 & 2.881e-07  & 2.00 \\
 \hline 
\end{tabular}
\end{center}
\end{table}

\small

\vspace{\baselineskip}

\color{black}
Before we end this section, we include one more numerical example that shows the relation between the order of convergence and the Fourier mode $n$. 
We consider the case when $k=1$ and investigate the order of convergence of the $\sigma$ approximation for Fourier modes $n$ from $2$ to $60$ in increments of $2$. 
We fix $\bff = (0,r,0)^T$ throughout this experiment.
The exact solution for various $n$-values are unknown for this case, so we compute the $L^2_r(\Omega)$-error between two consecutive approximations. 
Figure~\ref{graph2} shows the order of convergence of $\|\sigma_l-\sigma_{l-1}\|_{L^2_r(\Omega)}$ at $l=7$ in relation to various Fourier modes. 
When $\bff$ is independent of $n$, 
the exact solution will depend on $n$, and it is noticeable from Figure~\ref{graph2} that the order of convergence slowly decreases as $n$ gets larger. 

\begin{figure} 
\includegraphics[height=1.9in]{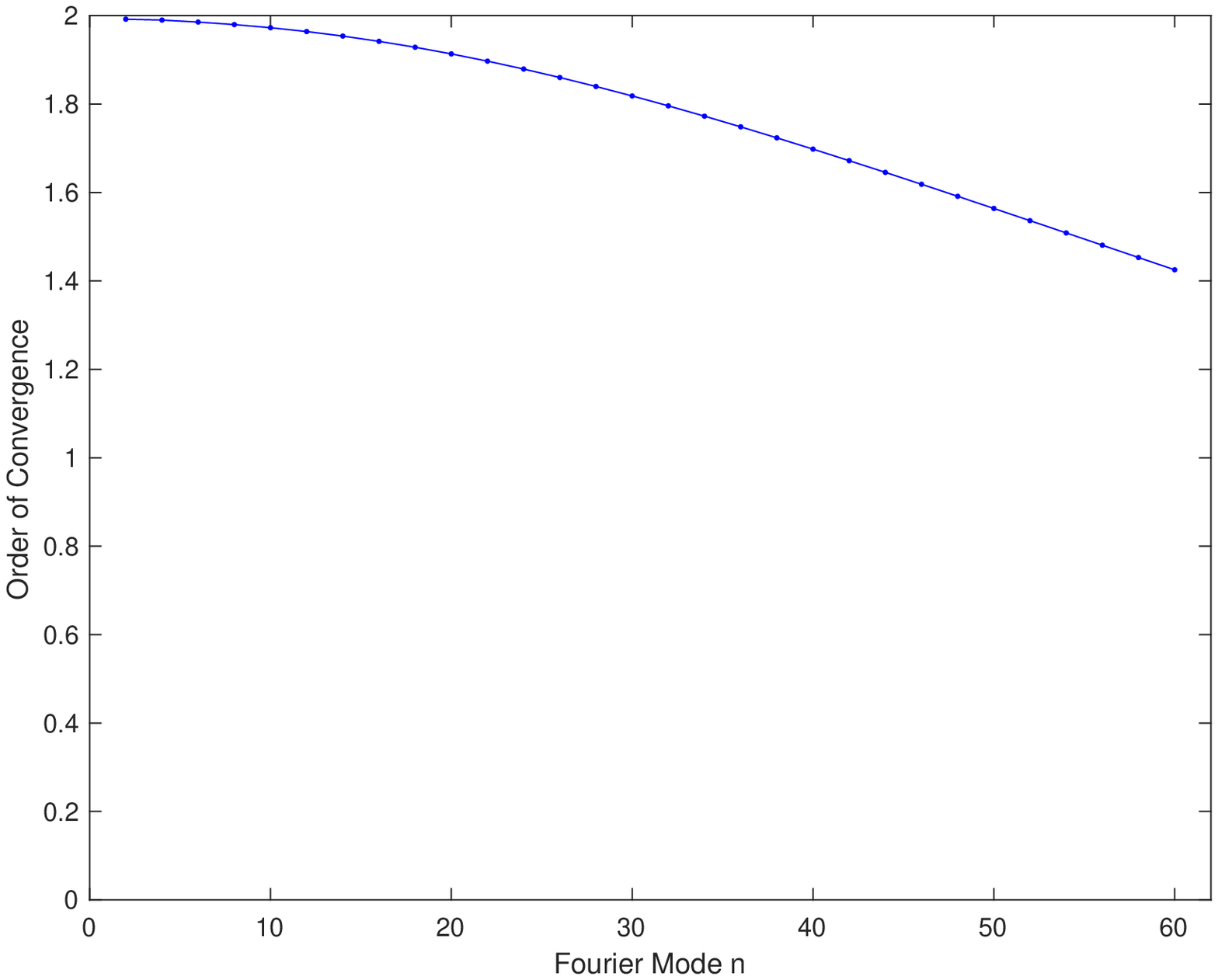}
\caption{The order of convergence of $\|\sigma_7-\sigma_{6}\|_{L^2_r(\Omega)}$ for various Fourier modes when $k=1$}
\label{graph2}
\end{figure}

\section{Concluding Remarks}

We studied the mixed formulation of the Hodge Laplacian on $3D$ axisymmetric domains with general data 
through Fourier-FEMs by using a recently developed family of Fourier finite element spaces. While we focused on 
axisymmetric $3D$ domains in this paper, the extension of this work to general $m$-dimensions ($mD$) will be 
interesting. Understanding what it means for a differential form to be axisymmetric in general $mD$
and extending the results in this paper to general $mD$ manifolds remain as future work.

\color{black}

\section{Appendix} \label{appendix}

\subsection{Understanding Axisymmetric Problems with General Data through Differential Forms} \label{differential_forms}

One way to reach the formulas of (\ref{formulas}) is by considering differential forms that represent a Fourier series decomposition and then applying the exterior derivative to it. 
Then, we can calculate the co-derivative in this setting to reach the formulas of the dual operators (\ref{dualoperator}).
In \cite{O:2014}, a new approach in analyzing axisymmetric problems with axisymmetric data (data that is independent of $\theta$) by using differential forms in a similar way was given. In this section, we will extend that approach to the 
general framework of axisymmetric problems with general data. Recall that $\breve{\Omega}\subset\RRR^3$ is an axisymmetric domain that can be obtained by rotating $\Omega\subset\RRR^2_+$ around the axis of symmetry.
We will use the standard notation in differential geometry that can be found in \cite{AFW:2010} for example. Since we will be using cylindrical coordinates, we will use 
$dr, rd\theta$, and $dz$ instead of $dx_1, dx_2$, and $dx_3$ when writing a differential $s$-form.

In \cite{O:2014}, the closed subspace of  $L^2\Lambda^s(\breve{\Omega})$ denoted by $\breve{L}^2\Lambda^s(\breve{\Omega})$
consisting of axisymmetric differential $s$-forms were considered. We write out the definition of $\breve{L}^2\Lambda^s(\breve{\Omega})$ here:
\begin{align*}
\breve{L}^2\Lambda^0(\breve{\Omega}) &= \{f\in L^2\Lambda^0(\breve{\Omega}): \dfrac{\partial f}{\partial\theta}=0 \}, \\
\breve{L}^2\Lambda^1(\breve{\Omega}) &= \{fdr+grd\theta+hdz\in L^2\Lambda^1(\breve{\Omega}): \dfrac{\partial f}{\partial\theta}=\dfrac{\partial g}{\partial\theta}=\dfrac{\partial h}{\partial\theta}=0 \}, \\
\breve{L}^2\Lambda^2(\breve{\Omega}) &= \{f rd\theta\wedge dz+gdz\wedge dr + hdr\wedge rd\theta\in L^2\Lambda^2(\breve{\Omega}):  \dfrac{\partial f}{\partial\theta}=\dfrac{\partial g}{\partial\theta}=\dfrac{\partial h}{\partial\theta}=0\}, \\
\breve{L}^2\Lambda^3(\breve{\Omega}) &= \{f dr\wedge rd\theta\wedge dz\in L^2\Lambda^3(\breve{\Omega}): \dfrac{\partial f}{\partial\theta}=0 \}.
\end{align*} 
These are the differential forms of interest when studying axisymmetric problems with axisymmetric data, and since if $\dfrac{\partial f}{\partial\theta}=0$ then
\[
\int\int\int_{\breve{\Omega}} f dV = 2\pi\int\int_\Omega f rdrdz,
\]
we utilized the weighted differential $t$-form space denoted by $L^2_r\Lambda^t(\Omega)$ for $t=0,1,2$ that consist of differential $t$-forms whose each component is 
square integrable with the weight $r$ (with the measure $rdrdz$).
This is again a Hilbert space with the inner product 
\begin{align*}
<\omega, \eta>_{L^2_r\Lambda^t} = \int_{\Omega} <\omega_x, \eta_x>_{\mathop\mathrm{Alt}^t T_x\Omega} rdrdz,
\end{align*} 
for all $\omega, \eta \in L^2_r\Lambda^t(\Omega)$. 

In this paper, we are considering axisymmetric problems with data that does have $\theta$-dependency, and therefore, we need to consider a subspace of $L^2\Lambda^s(\breve{\Omega})$
that is different from $\breve{L}^2\Lambda^s(\breve{\Omega})$. In particular, we need the following subspace of $L^2\Lambda^s(\breve{\Omega})$ that represent a Fourier series decomposition:
\begin{equation} \label{L2n_def}
 \begin{aligned}
 \tilde{L}^{2,n}\Lambda^0 &= \left\{ f\cos (n\theta)\in L^2\Lambda^0(\breve{\Omega}): \dfrac{\partial f}{\partial\theta}=0 \right\}, \\ 
 \tilde{L}^{2,n}\Lambda^1 &=  \left\{ f\cos(n\theta) dr + g\sin(n\theta) rd\theta + h\cos(n\theta) dz \in L^2\Lambda^1(\breve{\Omega}): \dfrac{\partial f}{\partial\theta}=\dfrac{\partial g}{\partial\theta}=\dfrac{\partial h}{\partial\theta}=0  \right\}, \\
\tilde{L}^{2,n}\Lambda^2 &= \left\{ f\sin(n\theta) rd\theta\wedge dz + g\cos(n\theta) dz\wedge dr + h\sin(n\theta) dr\wedge rd\theta\in L^2\Lambda^2(\breve{\Omega}):  \right. \\
 &\qquad\left. \dfrac{\partial f}{\partial\theta}=\dfrac{\partial g}{\partial\theta}=\dfrac{\partial h}{\partial\theta}=0 \right\}, \\
  \tilde{L}^{2,n}\Lambda^3 &= \left\{ f\sin(n\theta)dr\wedge rd\theta\wedge dz \in L^2\Lambda^3(\breve{\Omega}): \dfrac{\partial f}{\partial\theta}=0 \right\}.
 \end{aligned}
\end{equation}
The inner-product associated with these spaces is simply
\[
<w,v>_{\tilde{L}^{2,n}\Lambda^s} = <w,v>_{L^2\Lambda^s}.
\]
Note that if $\dfrac{\partial f}{\partial\theta}=0$ then
\begin{equation} \label{cos_sin}
\begin{aligned}
\int\int\int_{\breve{\Omega}} f^2\cos^2(n\theta) rdrd\theta dz &= \pi\int\int_\Omega f^2 rdrdz, \\
\int\int\int_{\breve{\Omega}} f^2\sin^2(n\theta) rdrd\theta dz &= \pi\int\int_\Omega f^2 rdrdz,
\end{aligned}
\end{equation}
so the functions $f, g$, and $h$ appearing in (\ref{L2n_def}) satisfy $f, g, h\in L^2_r\Lambda^0 (\Omega)$. 
Therefore, 
$<w,v>_{\tilde{L}^{2,n}\Lambda^s}$ can also be expressed by using the inner-product in $L^2_r\Lambda^0 (\Omega)$ as well. 

Next, consider the exterior derivative $d$ defined as
\begin{equation} \label{ex_der1}
d(\alpha_\sigma dx_{\sigma (1)}\wedge\cdot\cdot\cdot\wedge dx_{\sigma (s)}) = \sum_{j=1}^3 \frac{\partial\alpha_\sigma}{\partial x_j}dx_j\wedge dx_{\sigma (1)}\wedge\cdot\cdot\cdot\wedge dx_{\sigma (s)},
\end{equation} 
and the relations
\begin{equation} \label{formula1}
\begin{aligned} 
dx_1 &= \cos\theta dr - r\sin\theta d\theta, \\
dx_2 &= \sin\theta dr + r\cos\theta d\theta, \\
dx_3 &= dz,
\end{aligned}
\end{equation}
and
\begin{equation} \label{formula2}
\begin{aligned} 
dr &= \cos\theta dx_1 + \sin\theta dx_2, \\
rd\theta &= -\sin\theta dx_1 + \cos\theta dx_2, \\
dz &= dx_3.
\end{aligned}
\end{equation}
One can calculate
\begin{equation} \label{dn}
 \begin{aligned}
 &d(f\cos(n\theta)) \\
 &\quad= (\partial_rf)\cos(n\theta)dr - (\dfrac{n}{r})\sin(n\theta)rd\theta + (\partial_zf)\cos(n\theta)dz, \\
 &d(f\cos(n\theta)dr + g\sin(n\theta)rd\theta + h\cos(n\theta)dz) \\ 
 &\quad= (-\partial_zg -\dfrac{n}{r}h)\sin(n\theta)rd\theta\wedge dz 
 + (\partial_zf-\partial_rh)\cos(n\theta)dz\wedge dr 
 +(\partial_rg + \dfrac{n}{r}\cdot f + \dfrac{g}{r})\sin(n\theta) dr\wedge rd\theta, \\
 &d(f\sin(n\theta)rd\theta\wedge dz + g\cos(n\theta)dz\wedge dr + h\sin(n\theta)dr\wedge rd\theta) \\
 &\quad= (\partial_r f + \dfrac{f}{r} - \dfrac{n}{r}\cdot g + \partial_z h)\sin(n\theta) dr\wedge rd\theta\wedge dz.
 \end{aligned}
\end{equation}
Then we define the following spaces:
\begin{equation*}
\tilde{H}_{n}\Lambda^s = \left\{\omega\in  \tilde{L}^{2,n}\Lambda^s: d\omega\in  \tilde{L}^{2,n}\Lambda^{s+1} \right\}.  
\end{equation*}
This is a Hilbert space with the inner product being 
\begin{equation*}
<\omega,\nu>_{\tilde{H}_{n}\Lambda^s} = <\omega,\nu>_{ \tilde{L}^{2,n}\Lambda^s} + <d\omega,d\nu>_{ \tilde{L}^{2,n}\Lambda^{s+1}},
\end{equation*}
and we have a de Rham complex 
\begin{equation} \label{Lr2n_complex_cts}
0 \rightarrow \tilde{H}_{n}\Lambda^0 \xrightarrow{d} \tilde{H}_{n}\Lambda^1 \xrightarrow{d} \tilde{H}_{n}\Lambda^2 \xrightarrow{d} \tilde{H}_{n}\Lambda^3 \rightarrow 0.
\end{equation}

Let $inc: \partial\breve{\Omega}\rightarrow\breve{\Omega}$ be the inclusion map. Then, the pullback of the map $inc$ is the trace map 
denoted by $\mathrm{tr}: \Lambda^s(\breve{\Omega})\rightarrow\Lambda^s(\partial\breve{\Omega})$, and it is continuous from $\tilde{H}_{n}\Lambda^s(\breve{\Omega})$ to $H^{-1/2}\Lambda^s(\partial\breve{\Omega})$.

On $\breve{\Omega}$, the Hodge star operator $\star$ from $\Lambda^s(\breve{\Omega})$ to $\Lambda^{3-s}(\breve{\Omega})$ satisfies
\[
\int_{\breve{\Omega}} \omega\wedge\mu = <\star\omega, \mu>_{\tilde{L}^{2,n}\Lambda^{3-s}}
\]
for all $\mu\in \tilde{L}^{2,n}\Lambda^{3-s}$. 

The coderivative operator $\delta$ maps a $s$-form to a $(s-1)$-form in the following way:
\begin{equation} \label{coderivative}
\star (\delta\omega) = (-1)^sd (\star\omega).
\end{equation}

Now, by using the definition (\ref{coderivative}), let us calculate the co-derivative $\delta$ of $d$. %Reference previous section with definitions here.
\begin{equation} \label{delta_n}
 \begin{aligned}
  &\delta( f\cos(n\theta) dr + g\sin(n\theta) rd\theta + h\cos(n\theta) dz) \\
   &\quad= -(\partial_rf+\dfrac{f}{r}+\dfrac{n}{r}\cdot g + \partial_zh)\cos(n\theta), \\
   &\delta(f\sin(n\theta)rd\theta\wedge dz + g\cos(n\theta)dz\wedge dr + h\sin(n\theta)dr\wedge rd\theta) \\
   &\quad= (-\partial_zg +\dfrac{n}{r}h)\cos(n\theta)dr + (\partial_zf-\partial_rh)\sin(n\theta)rd\theta +(\partial_rg - \dfrac{n}{r}\cdot f + \dfrac{g}{r})\cos(n\theta) dz, \\ 
   &\delta(fdr\wedge rd\theta\wedge dz) \\
   &\quad= (-\partial_rf)\sin(n\theta)rd\theta\wedge dz - (\dfrac{nf}{r})\cos(n\theta)dz\wedge dr - (\partial_zf)\sin(n\theta)dr\wedge rd\theta.
 \end{aligned}
\end{equation}
As we did for the exterior derivative, we define 
\begin{equation*}
\tilde{H}_{n}^*\Lambda^s = \left\{\omega\in \tilde{L}^{2,n}\Lambda^s: \delta\omega\in \tilde{L}^{2,n}\Lambda^{s-1} \right\}.  
\end{equation*}
Furthermore, we define the subspace with vanishing trace:
\begin{equation*}
\tilde{H}_{n,0}^*\Lambda^s= \left\{ \omega\in \tilde{H}_{n}^*\Lambda^s: \mathrm{tr}\tilde{\omega}=0  \right\}.
\end{equation*}
Its inner product is
\begin{equation*}
<\omega,\nu>_{\tilde{H}_{n}^*\Lambda^s} = <\omega,\nu>_{ \tilde{L}^{2,n}\Lambda^s} + <\delta\omega,\delta\nu>_{ \tilde{L}^{2,n}\Lambda^{s-1}},
\end{equation*}
and we have a dual complex 
\begin{equation} \label{Lr2n_complex_cts}
0 \leftarrow \tilde{H}^*_{n,0}\Lambda^0 \xleftarrow{\delta} \tilde{H}^*_{n,0}\Lambda^1\xleftarrow{\delta} \tilde{H}^*_{n,0}\Lambda^2 \xleftarrow{\delta} \tilde{H}^*_{n,0}\Lambda^3 \leftarrow 0.
\end{equation}

By the Leibniz rule, Stokes theorem, and the fact that pullbacks respect the wedge product, we obtain the integration by parts formula for differential forms (\cite[section 4]{AFW:2010}):
\[
<d\omega, \mu>_{\tilde{L}^{2,n}\Lambda^s} = <\omega,\delta\mu>_{\tilde{L}^{2,n}\Lambda^{s-1}} + \int_{\partial\breve{\Omega}} \mathrm{tr}\omega \wedge \star\mu \tn{ for all } \omega\in \Lambda^{s-1}(\breve{\Omega}), \mu\in \Lambda^{s}(\breve{\Omega}),
\]
and this can be extended to 
\[
<d\omega, \mu>_{\tilde{L}^{2,n}\Lambda^s} = <\omega,\delta\mu>_{\tilde{L}^{2,n}\Lambda^{s-1}} \tn{ for all } \omega\in \tilde{H}_{n}\Lambda^{s-1}, \mu\in \tilde{H}_{n,0}^*\Lambda^s.
\]
Furthermore, the following theorem follows. (See \cite[Theorem 4.1]{AFW:2010}.)

\begin{theorem} \label{adjoint}
Let $d$ be the exterior derivative viewed as an unbounded operator from $\tilde{L}^{2,n}\Lambda^{s-1}$ to $\tilde{L}^{2,n}\Lambda^s$ with domain $\tilde{H}_{n}\Lambda^{s-1}$.
Then, the adjoint of $d$, as an unbounded operator from $\tilde{L}^{2,n}\Lambda^s$  to $\tilde{L}^{2,n}\Lambda^{s-1}$, has $\tilde{H}_{n,0}^*\Lambda^s$ as its domains 
and coincides with $\delta$.
\end{theorem}

\subsection{Proof of Proposition \ref{prop:inverse} item (\ref{inv4})} \label{inverse_proof}

Recall that $r_K$ denotes maximal value of $r$ in $K$. For triangles that do not intersect the $z$-axis, the proof of (\ref{inv4}) is trivial, since
\[
r_K^2\|\dfrac{u}{r}\|^2_{L^2_r(K)} \leq \dfrac{r_K^2}{r_{K,min}^2}\|u\|^2_{L^2_r(K)} \leq C\|u\|^2_{L^2_r(K)},
\]
where $r_{K,min}$ denotes the minimal value of $r$ in $K$.
This is true, since $\dfrac{r_K}{r_{K,min}}\leq C$ for triangles that are aways from the $z$-axis.

For triangles that do intersect the $z$-axis, we use \cite[Corollary 4.1]{MR:1982} and \cite[Lemma 10]{Lacoste:2000} that states that
\begin{equation} \label{lem1}
\int_K \dfrac{u^2}{r^2} drdz \leq C\int_K |\gradrz u|^2 drdz,
\end{equation}
for all function $u\in H^1(K)$ that vanishes on $\Gamma_0$.

By using (\ref{lem1}), the proof of (\ref{inv4}) for any triangle $K$ that intersects the $z$ is completed by
\begin{equation*}
\begin{aligned}
r_K^2\|\dfrac{u}{r}\|^2_{L^2_r(K)} &= r_K^2\int_K \dfrac{u^2}{r} drdz \\ 
&= r_K^2\int_K \dfrac{u^2}{r^2} r drdz \\
&\leq r_K^3 \int_K \dfrac{u^2}{r^2}  drdz \\
&\leq Cr_K^3\int_K |\gradrz u|^2 drdz &&\textnormal{ by (\ref{lem1}) } \\
&\leq Cr_K^3h_K^2\|\gradrz u\|^2_{L^\infty(K)} \\
&\leq Cr_K^2\| \gradrz u \|^2_{L^2_r(K)} &&\textnormal{ by (\ref{inv3}) } \\
&\leq C\dfrac{r_K^2}{h_K^2} \| u \|^2_{L_r^2(K)} &&\textnormal{ by (\ref{inv1}) } \\
&\leq C\| u \|^2_{L_r^2(K)} &&\textnormal{ since $r_K\leq Ch_K$ on $K$ that intersects the $z$-axis. }
\end{aligned}
\end{equation*}
This completes the proof of (\ref{inv4}).

\subsection{Proof of Proposition \ref{prop:eta}} \label{prop_proof}

We will prove this proposition in two separate cases as done in \cite{GO:2012}. 

The first case is when $\ba$ is a point on the $z$-axis, i.e., $\ba=(0,a_z)$. Then $D_\ba$ is the half disk $\{(r,z)\in\RRR^2_+: r^2+(z-a_z)^2\leq\rho^2 \}$.
Let $\hat{D}_1=\{(\hat{r},\hat{z}):\hat{r}^2+\hat{z}^2\leq 1, \hat{r}\geq 0\}$ be the reference domain, and consider the mapping 
\[
r=\hat{r}\rho \quad\quad z=\hat{z}\rho+a_z.
\] 
On the reference domain, define $\hat{\eta}_1\in \hat{P}_l$, where $\hat{P}_l$ is the space of polynomials on $\hat{D}_1$ of order less than or equal to $l$, in the following way:
\[
\int_{\hat{D}_1} \hat{r}^3\hat{\eta}_1 \hat{p}  d\hat{r}d\hat{z} = \hat{p}(\mathbf{0}) \quad\forall \hat{p}\in \hat{P}_l.
\]
Set $\eta_\ba(r,z)=\dfrac{1}{\rho^5}\hat{\eta}_1(\hat{r},\hat{z})$. Then
\begin{align*}
(r\eta_\ba,rp)_{L^2_r(D_\ba)} & =\int_{D_\ba} r^3\eta_\ba(r,z) p(r,z)drdz \\
&= \int_{\hat{D}_1} \hat{r}^3\rho^3 \dfrac{1}{\rho^5}\hat{\eta}_1(\hat{r},\hat{z})\hat{p}(\hat{r},\hat{z}) \rho^2d\hat{r}d\hat{z} \\ 
&= \int_{\hat{D}_1} \hat{r}^3\hat{\eta}_1 \hat{p} d\hat{r}d\hat{z} \\ 
&= \hat{p}(\mathbf{0}) = p(\ba),
\end{align*}
where $\hat{p}(\hat{r},\hat{z})=p(r,z)$. This proves (\ref{eta1}). Furthermore,
\[
\|r\eta_\ba\|_{L^2_r(D_\ba)}^2 = \eta_\ba(\ba) = \dfrac{1}{\rho^5} \hat{\eta}_1(\mathbf{0}) \leq \dfrac{C}{\rho^2 r_a^3}, 
\]
since $r_\ba=\rho$ in case 1. This proves (\ref{eta1}). 
Furthermore, (\ref{eta2}) holds true as
\begin{align*}
\int_{D_\ba} r^3(\by)|\eta_{\ba}(\by)| d\boldsymbol{y} 
&= \int_{D_\ba} |\eta_{\ba}r^{3/2}| r^{3/2} d\boldsymbol{y} \\
&\leq \|r^{3/2}\eta_\ba\|_{L^2(D_\ba)}\|r^{3/2}\|_{L^2(D_\ba)} &&\textnormal{ by Holder's inequality } \\
&= \|r\eta_{\ba}\|_{L^2_r(D_\ba)}\sqrt{\int_{D_\ba} r^3 d\boldsymbol{y}} \\
&\leq \dfrac{C}{\sqrt{\rho^2r_\ba^3}}\sqrt{r_{a, max}^3 \cdot\pi\rho^2}  &&\textnormal{ by (\ref{eta1}) } \\
&\leq C,
\end{align*}
where $r_{a, max}=\max_{\bx\in D_\ba} r(\bx)$, so $\dfrac{r_{a, max}}{\rho}=1$. This proves (\ref{eta2}) and (\ref{eta3}) for the first case.

Next, we consider the second case where $\ba=(a_r, a_z)$ is not on the $z$-axis. In this case, $D_\ba=\{(r,z)\in\RRR^2_+: (r-a_r)^2+(z-a_z)^2\leq\rho^2 \}$,
and we will used the reference domain  $\hat{D}_2=\{(\hat{r},\hat{z}):\hat{r}^2+\hat{z}^2\leq 1\}$. 
The mapping from that we will be using between $D_\ba$ and $\hat{D}_2$ is
\[
r=a_r+\rho\hat{r} \quad\quad z=a_z+\rho\hat{z}.
\]
Let us consider two polynomials $\hat{\eta}_{\ba,1}\in\hat{P}_l$ and 
$\hat{\eta}_{\ba,\star}\in\hat{P}_l$ on $\hat{D}_2$ defined as
\begin{align*}
\int_{\hat{D}_2} (a_r+\rho\hat{r})^3\hat{\eta}_{\ba,\star}\hat{p} d\hat{r} d\hat{z} &= \hat{p}(\mathbf{0}) &&\textnormal{ for all } \hat{p}\in \hat{P}_l, \\
\int_{\hat{D}_2} \hat{\eta}_{\ba,1}\hat{p} d\hat{r} d\hat{z} &= \hat{p}(\mathbf{0}) &&\textnormal{ for all } \hat{p}\in \hat{P}_l.
\end{align*} 
Note that, since $r=a_r+\rho\hat{r}$ and $D_\ba$ is away from the $z$-axis, $a_r+\rho\hat{r}$ is bounded above and below on $\hat{D}_2$.
Then, since
\begin{equation*} 
\begin{aligned}
\int_{\hat{D}_2} (a_r+\rho\hat{r})^3\hat{\eta}_{\ba,\star}^2 d\hat{r} d\hat{z} &= \hat{\eta}_{\ba,\star}(\mathbf{0}) \\
&= \int_{\hat{D}_2} \hat{\eta}_{\ba,1}  \hat{\eta}_{\ba,\star} d\hat{r} d\hat{z} \\
&\leq \|\hat{\eta}_{\ba,1}\|_{L^2(\hat{D}_2)}  \|\hat{\eta}_{\ba,\star}\|_{L^2(\hat{D}_2)} \quad\textnormal{ by Holder's inequality } \\
&\leq \|\hat{\eta}_{\ba,1}\|_{L^2(\hat{D}_2)} \sqrt{\int_{\hat{D}_2} \hat{\eta}_{\ba,\star}^2  (a_r+\rho\hat{r})^3 d\hat{r} d\hat{z}} \cdot \max_{\by\in\hat{D}_2}\left[(a_r+\rho\hat{r})^{-3/2}\right]. \\
\end{aligned}
\end{equation*}
This shows that 
\[
\sqrt{\int_{\hat{D}_2} \hat{\eta}_{\ba,\star}^2  (a_r+\rho\hat{r})^3 d\hat{r} d\hat{z}} \leq C\max_{\by\in\hat{D}_2}\left[(a_r+\rho\hat{r})^{-3/2}\right],
\]
and so 
\begin{equation} \label{eqeq1}
\hat{\eta}_{\ba,\star}(\mathbf{0}) = \int_{\hat{D}_2} (a_r+\rho\hat{r})^3\hat{\eta}_{\ba,\star}^2 d\hat{r} d\hat{z} \leq  C\max_{\by\in\hat{D}_2}\left[(a_r+\rho\hat{r})^{-3}\right].
\end{equation}
Now let us define $\eta_\ba(r,z)$ on $D_\ba$ as
\[
\eta_\ba(r,z) = \dfrac{1}{\rho^2}\hat{\eta}_{\ba,\star}(\hat{r},\hat{z}). 
\]
Then, (\ref{eta1}) is true, since
\begin{align*}
(r\eta_a, rp)_{L^2_r(D_\ba)} &= \int_{D_\ba} r^3\eta_a(r,z)p(r,z) drdz \\
&= \int_{\hat{D}_2}  (a_r+\rho\hat{r})^3\dfrac{1}{\rho^2}\hat{\eta}_{\ba,\star}(\hat{r},\hat{z})\hat{p}(\hat{r},\hat{z}) \rho^2d\hat{r} d\hat{z} \\
&= \int_{\hat{D}_2}  (a_r+\rho\hat{r})^3 \hat{\eta}_{\ba,\star}(\hat{r},\hat{z}) \hat{p}(\hat{r},\hat{z})d\hat{r} d\hat{z} \\
&= \hat{p}(\mathbf{0}) = p(\ba) &&\textnormal{ for all } p\in P_l.
\end{align*}
Furthermore, (\ref{eta2}) and (\ref{eta3}) also hold true, since
\begin{align*}
\|r\eta_\ba\|_{L^2_r(D_\ba)}^2 = \eta_\ba(\ba) = \dfrac{1}{\rho^2}\hat{\eta}_{\ba,\star}(\mathbf{0}) \leq  \dfrac{C}{\rho^2r_\ba^3},
\end{align*}
by (\ref{eqeq1}), and 
\begin{align*}
\int_{D_\ba} r^3(\by)|\eta_{\ba}(\by)| d\boldsymbol{y} 
&= \int_{D_\ba} |\eta_{\ba}(\by)r^{3/2}| r^{3/2} d\boldsymbol{y} \\
&\leq \|r\eta_{\ba}\|_{L^2_r(D_\ba)}\sqrt{\int_{D_\ba} r^3 d\boldsymbol{y}} &&\textnormal{ by Holder's inequality } \\
&\leq \dfrac{C}{\sqrt{\rho^2r_\ba^3}}\sqrt{r_{a, max}^3 \cdot\pi\rho^2}  &&\textnormal{ by (\ref{eta1}) } \\
&\leq C &&\tn{ by (\ref{threethreethree}), }
\end{align*}
and this completes the proof. 
\bibliographystyle{siam}	
\bibliography{reference}	

\end{document}